\documentclass[a4paper]{amsart}

\pdfoutput=1

\usepackage{amsmath,amsthm,amssymb}
\usepackage{mathtools}
\usepackage{mathrsfs}

\usepackage{microtype}
\usepackage{lmodern}
\usepackage[T1]{fontenc}
\usepackage[english]{babel}

\usepackage[autostyle]{csquotes}

\usepackage{booktabs}

\usepackage{varwidth}

\usepackage{tikz-cd}

\usepackage[a4paper,inner=3cm,outer=3cm,top=3.3cm,bottom=3cm]{geometry}
\usepackage[colorlinks=true,citecolor=blue, urlcolor=magenta, breaklinks=true]{hyperref}

\newtheorem{thm}{Theorem}[section]
\newtheorem{lemma}[thm]{Lemma}

\newtheorem{prop}[thm]{Proposition}

\newtheorem{cor}[thm]{Corollary}
\theoremstyle{definition}
\newtheorem{defn}[thm]{Definition}
\newtheorem{rem}[thm]{Remark}
\newtheorem{example}[thm]{Example}

\numberwithin{equation}{section}

\DeclareMathOperator{\Div}{div}

\DeclareMathOperator*{\im}{im}
\DeclareMathOperator{\lspan}{span}
\DeclareMathOperator{\rk}{rk}
\DeclareMathOperator*{\Supp}{supp}

\newcommand{\cone}{{\mathrm{cone}}}
\newcommand{\Hom}{{\mathrm{Hom}}}
\newcommand{\PGL}{{\mathrm{PGL}}}
\newcommand{\SL}{{\mathrm{SL}}}
\newcommand{\Spin}{\mathrm{Spin}}

\newcommand{\rleft}{\mathopen{}\mathclose\bgroup\left}
\newcommand{\rright}{\aftergroup\egroup\right}

\newcommand{\Cd}{{\mathbb{C}}}
\newcommand{\Qd}{{\mathbb{Q}}}

\newcommand{\Zd}{{\mathbb{Z}}}

\newcommand{\C}{\Cd}

\newcommand{\Q}{\Qd}

\newcommand{\Z}{\Zd}
\newcommand{\X}{\Xf}

\newcommand{\Cm}{{\mathcal{C}}}
\newcommand{\Dm}{{\mathcal{D}}}

\newcommand{\Hm}{{\mathcal{H}}}
\newcommand{\Mm}{{\mathcal{M}}}
\newcommand{\Nm}{{\mathcal{N}}}
\newcommand{\Pm}{{\mathcal{P}}}
\newcommand{\Sm}{{\mathcal{S}}}
\newcommand{\Vm}{{\mathcal{V}}}

\newcommand{\Xf}{{\mathfrak{X}}}

\newcommand{\Ss}{{\mathscr{S}}}

\newcommand{\Bt}{{\widetilde{B}}}
\newcommand{\Dmt}{{\widetilde{\Dm}}}
\newcommand{\Dt}{{\widetilde{D}}}
\newcommand{\Gt}{{\widetilde{G}}}
\newcommand{\Ht}{{\widetilde{H}}}
\newcommand{\iotat}{{\widetilde{\iota}}}
\newcommand{\Mmt}{{\widetilde{\Mm}}}
\newcommand{\Nmt}{{\widetilde{\Nm}}}
\newcommand{\Phit}{{\widetilde{\Phi}}}
\newcommand{\psit}{{\widetilde{\psi}}}

\newcommand{\rhot}{{\widetilde{\rho}}}
\newcommand{\Sigmat}{{\widetilde{\Sigma}}}
\newcommand{\Smt}{{\widetilde{\Sm}}}
\newcommand{\St}{{\widetilde{S}}}
\newcommand{\Tt}{{\widetilde{T}}}
\newcommand{\varthetat}{{\widetilde{\vartheta}}}
\newcommand{\Vmt}{{\widetilde{\Vm}}}

\newcommand{\with}{\colon}
\newcommand{\set}[1]{\rleft\{{#1}\rright\}}

\makeatletter
\@ifpackageloaded{hyperref}{%
\DeclareRobustCommand{\SkipTocEntry}[5]{}}{%
\DeclareRobustCommand{\SkipTocEntry}[4]{}}
\makeatother

\begin{document}
\selectlanguage{english}

\title{Containment Relations among Spherical Subgroups}

\author{Johannes Hofscheier}
\address{Department of Mathematics and Statistics\\ McMaster University\\ 1280 Main Street West\\ Hamilton, Ontario L8S4K1\\ Canada}
\email{johannes.hofscheier@math.mcmaster.ca}


\begin{abstract}
  A closed subgroup $H$ of a connected reductive group $G$ is called \emph{spherical} if a Borel subgroup in $G$ has an open orbit on $G/H$. We give a combinatorial characterization for a spherical subgroup to be contained in another one which generalizes previous work by Knop. As an application, we compute the Luna datum of the identity component of a spherical subgroup which yields a characterization of connectedness for spherical subgroups.
\end{abstract}

\maketitle

\section{Introduction}
\label{sec:introduction}
Throughout the paper, we work with algebraic varieties and algebraic groups over the field of complex numbers $\C$.

Let $G$ be a connected reductive complex algebraic group. A closed subgroup $H \subseteq G$ is called \emph{spherical} if $G/H$ has a dense open orbit for a Borel subgroup $B \subseteq G$. In this case, $G/H$ is called a \emph{spherical homogeneous space}. An equivariant open embedding of $G/H$ into a normal irreducible $G$-variety $X$ is called a \emph{spherical embedding}, and $X$ is called a \emph{spherical variety}.

The Luna conjecture \cite{Luna:TypeA}, which has been proved by several researchers (see \cite{BP:PrimitiveWonderful,CF:WonderfulVar,Losev:Unique}), provides a combinatorial description of the spherical subgroups of a given connected reductive group $G$. In this work, we give a description of the set $\Hm_H$ of spherical subgroups of $G$ which contain a fixed spherical group $H \subseteq G$ in terms of the combinatorial objects of the Luna conjecture. This extends \cite[Section 4]{Knop:LV} where Knop gives a combinatorial description of the subset $\Hm_H^{\mathsf{c}} \subseteq \Hm_H$ of \emph{co-connected} spherical subgroups $\Ht \subseteq G$ containing $H$, i.e., $\Ht/H$ is connected. The set $\Hm_H$ plays an important role in the study of equivariant morphisms between spherical varieties (see, for instance, \cite[Section 4]{Knop:LV}). Indeed, a second spherical subgroup $H' \subseteq G$ contains $H$ (up to conjugation) if and only if there is a $G$-equivariant morphism $G/H \to G/H'$.

\subsection{Preliminaries}
\label{sec:prelim}

Before stating our main result, let us recall the combinatorial description of spherical subgroups: Fix a Borel subgroup $B \subseteq G$ and a maximal torus $T \subseteq B$, denote by $R$ the associated root system in the character lattice $\X(T)$, and write $S \subseteq R$ for the set of simple roots corresponding to $B$. The character lattices $\X(B)$ and $\X(T)$ are naturally identified via restricting characters from $B$ to $T$. Luna assigned to any spherical homogeneous space $G/H$ its \emph{Luna datum} (or \emph{homogeneous spherical datum}) $\Sm \coloneqq ( \Mm , \Sigma , S^p, \Dm^a )$, which, by a theorem of Losev (see \cite{Losev:Unique}), uniquely determines $G/H$ up to $G$-equivariant isomorphisms. Let us describe the combinatorial invariants in $\Sm$.

$\Mm$ denotes the weight lattice of $B$-semi-invariant rational functions $f$ on $G / H$, i.e., $f \in \C(G/H)^{(B)}$, and a rational function $f_\chi \in \C( G / H )^{(B)}$ is uniquely determined (up to a constant factor) by its $B$-weight $\chi \in \X(B)$. Let $\Nm \coloneqq \Hom( \Mm, \Z )$ be the dual lattice together with the natural pairing $\langle \cdot, \cdot \rangle \colon \Nm \times \Mm \to \Z$.

$\Vm$ denotes the set of $G$-invariant discrete valuations $\nu \colon \C( G / H )^* \to \Q$ and the assignment $\langle \nu, \chi \rangle \coloneqq \nu(f_\chi)$ for $\nu \in \Vm$ induces an inclusion $\Vm \subseteq \Nm_\Q \coloneqq \Nm \otimes \Q$ (see \cite[7.4, Proposition]{LunaVust}), where $\Vm$ is known to be a cosimplicial cone (see~\cite{Brion:GenEspSym}) called the \emph{valuation cone}. The set of \emph{spherical roots} $\Sigma$ of $G/H$ consists of the primitive generators in $\Mm$ of the extremal rays of the negative of the dual of the valuation cone $\Vm$.

The stabilizer $P$ of the open $B$-orbit in $G/H$ is a parabolic subgroup of $G$ containing $B$, hence uniquely determines a subset $S^p \subseteq S$ of simple roots.

We write $\Dm$ for the set of \emph{colors} of $G/H$, i.e., the set of $B$-invariant prime divisor in $G/H$. To any $D \in \Dm$, we associate the element $\rho(D) \in \Nm$ defined by $\langle \rho(D), \chi \rangle \coloneqq \nu_D( f_\chi )$ where $f_\chi \in \C( G / H )^{(B)}$ is a $B$-semi-invariant rational function of weight $\chi \in \Mm$ and $\nu_D$ denotes the valuation induced by $D$. For any $\alpha \in S$, there is a minimal parabolic subgroup $P_\alpha \subseteq G$ containing $B$ and corresponding to $\alpha$. We introduce the set of colors $\Dm( \alpha ) \subseteq \Dm$ \emph{moved} by $P_\alpha$, i.e., the colors $D \in \Dm$ such that $P_\alpha \cdot D \neq D$. We obtain a map $\varsigma \colon \Dm \to \Pm( S )$ which assigns to a color $D$ the set of simple roots $\alpha$ such that $P_\alpha$ moves $D$ ($\Pm(S)$ denotes the power set of $S$). The colors are divided into three types: $D \in \Dm(\alpha)$ is said to be of \emph{type $a$} if $\alpha \in \Sigma$, of \emph{type $2a$} if $2 \alpha \in \Sigma$, and otherwise it is said to be of \emph{type $b$}. The type does not depend on the choice of $\alpha \in \varsigma(D)$ and we obtain a disjoint union $\Dm = \Dm^a \cup \Dm^{2a} \cup \Dm^b$ where  $\Dm^a$, $\Dm^{2a}$ and $\Dm^b$ denote the sets of colors of the corresponding types. The fourth entry $\Dm^a$ of $\Sm$ is treated as an abstract set equipped with the map $\rho|_{\Dm^a}$.

Fix a spherical subgroup $H \subseteq G$ and let $\Sm = ( \Mm, \Sigma, S^p, \Dm^a)$ be its Luna datum. In \cite[Section 4]{Knop:LV}, Knop obtains a correspondence between $\Hm_H^{\mathsf{c}}$ and \emph{colored subspaces}, i.e., pairs $( \Nm^1_\Q, \Dm^1 )$ where $\Nm^1_\Q \subseteq \Nm_\Q$ is a subspace and $\Dm^1 \subseteq \Dm$ is a subset such that $\Nm^1_\Q$ coincides with the cone spanned by $\Vm \cap \Nm^1_\Q$ and $\rho(D)$ for all $D \in \Dm^1$. For any colored subspace $( \Nm^1_\Q, \Dm^1 )$, the Luna datum $\Sm_0 = ( \Mm_0, \Sigma_0, S^p_0, \Dm^a_0 )$ of the corresponding subgroup $H_0$ is given as follows (see, for instance, \cite[Proposition 3.4.3]{Losev:Unique})
\[
  \Mm_0 = ( \Nm^1_\Q )^\perp \cap \Mm, \; S^p_0 = \set{ \alpha \in S \with \Dm( \alpha ) \subseteq \Dm^1 }, \; \Dm^a_0 = \set{ D \in \Dm^a \with \varsigma(D) \cap \Sigma_0 \neq \emptyset } \text{,}
\]
where $\Sigma_0$ are the primitive ray generators in $\Mm_0$ of $\cone(\Sigma) \cap ( \Nm^1_\Q )^\perp$, and $\Dm^a_0$ is equipped with the map $\rho_0 \colon \Dm^a_0 \to \Hom( \Mm_0, \Z)\eqqcolon \Nm_0; D \mapsto ( \pi \circ \rho) (D)$ where $\pi \colon \Nm \to \Nm_0$ is the map dual to the inclusion $\iota \colon \Mm_0 \hookrightarrow \Mm$.

Recall from \cite[Definition 4.1.1]{Losev:Unique}, that the set of \emph{distinguished roots} $\Sigma^+$ consists of all $\gamma \in \Sigma$ such that, (1) $\gamma$ is in the \emph{root lattice} $\X(R)$, (2) there is a spherical subgroup in $G$ with Luna datum $( \Z 2\gamma, \set{ 2\gamma }, S^p, \emptyset )$, and, (3) if $\gamma \in S$, then $\rho(D^+) = \rho(D^-)$ where $\Dm(\gamma) = \set{ D^\pm }$.

\subsection{Main Results}
\label{sec:main-res}

\begin{defn}
  \label{defi:dist-pair}
  A pair $( \Mmt, \Dm^1 )$ of a subgroup $\Mmt \subseteq \Mm$ and a subset $\Dm^1 \subseteq \Dm$ is called \emph{distinguished} if $\Cm \coloneqq ( \Mmt^\perp, \Dm^1 )$ is a colored subspace, and $\tfrac{1}{2} \gamma \in \Sigma_0^+ \cup \rleft( \Sigma_0 \setminus \X(R) \rright)$ for every $\gamma \in \Sigmat \setminus \Sigma_0$ where $( \Mm_0, \Sigma_0, S^p_0, \Dm^a_0 )$ is the Luna datum of the spherical subgroup corresponding to $\Cm$ and $\Sigmat$ is the set of primitive generators in $\Mmt$ of $\cone( \Sigma ) \cap \Mmt_\Q$.
\end{defn}

\begin{thm}
  \label{thm:corrsp-dist-pairs-subgrps}
  Let $\Ht \in \Hm_H$ with Luna datum $\Smt = ( \Mmt, \Sigmat, \St^p, \Dmt^a )$. Let $\Dm^1 \subseteq \Dm$ be the subset of colors of $G/H$ which get mapped dominantly onto $G/\Ht$ under the natural map $G/H \to G/\Ht$. Then $( \Mmt, \Dm^1 )$ is a distinguished pair and the assignment $\Ht \mapsto ( \Mmt, \Dm^1 )$ induces a bijection between $\Hm_H$ and distinguished pairs.
\end{thm}

Theorem \ref{thm:corrsp-dist-pairs-subgrps} can be rephrased using Luna data which yields a generalization of a quotient system of spherical systems (see \cite[Section 2.3]{BL:F4}).

\begin{defn}
  \label{def:subdatum}
  A Luna datum $\Smt = ( \Mmt, \Sigmat, \St^p, \Dmt^a )$ is called a \emph{subdatum of $\Sm$}, if there exists a distinguished pair $( \Mmt, \Dm^1 )$ such that
  \begin{enumerate}
  \item $\Sigmat$ are the primitive ray generators in $\Mmt$ of $\cone(\Sigma) \cap \Mmt_\Q$,
  \item $\St^p = \set{ \alpha \in S \with \Dm( \alpha ) \subseteq \Dm^1 }$, and
  \item $\Dmt^a = \set{ D \in \Dm^a \with \varsigma( D ) \cap \Sigmat \neq \emptyset }$ equipped with the map $\rhot \colon \Dmt^a \to \Hom( \Mmt, \Z ) \eqqcolon \Nmt; D \mapsto ( \pi \circ \rho) (D)$ where $\pi \colon \Nm \to \Nmt$ is the map dual to the inclusion $\iota \colon \Mmt \hookrightarrow \Mm$.
  \end{enumerate}
  We obtain a partial ordering on the set of Luna data, namely $\Smt \preceq \Sm$ if $\Smt$ is a subdatum of $\Sm$.
\end{defn}

\begin{thm}
  \label{thm:corr-subdata-subgrp}
  The assignment $\Smt \mapsto \Ht$ which associates to a subdatum $\Smt \preceq \Sm$ the corresponding spherical subgroup $\Ht$ induces an order-reversing bijection from the set of subdata of $\Sm$ onto $\Hm_H/\sim$ where \enquote{$\sim$} denotes the equivalence relation induced by conjugation.
\end{thm}

\begin{cor}
  The Luna datum $\Sm^\circ = ( \Mm^\circ, \Sigma^\circ, (S^\circ)^p, (\Dm^\circ)^a )$ of the identity component $H^\circ$ of a spherical subgroup $H \subseteq G$ with Luna datum $\Sm$ is  given by
  \[
    \Mm^\circ = \set{ x \in \Mm_\Q \cap \X(B) \colon \langle \rho( \Dm ), x \rangle \subseteq \Z }, \quad (S^\circ)^p = S^p \text{,}
  \]
  and $\Sigma^\circ$ is the set of primitive generators in $\Mm^\circ$ of the extremal rays of $\cone( \Sigma )$. For every $\alpha \in ( \tfrac{1}{2} \Sigma ) \cap S \cap \Sigma^\circ$, we introduce pairwise distinct elements $D_\alpha^{\pm}$ and set $\rho^\circ( D_\alpha^{\pm} ) \coloneqq \tfrac{1}{2} \check{\alpha}|_{\Mm^\circ}$. Moreover, for $D \in \Dm^a$, we set $\rho^\circ(D)|_\Mm \coloneqq \rho(D)$ which is possible as $\Mm$ is a sublattice of $\Mm^\circ$ of finite index. Then $(\Dm^\circ)^a$ is the disjoint union of $\Dm^a$ and the elements $D_\alpha^\pm$ from above.
\end{cor}

\begin{defn}
  \label{def:d-sat}
  A lattice $\Mm' \subseteq \Mm$ is called \emph{$\Dm$-saturated}, if 
  \[
    \Mm' = \set{ x \in \Mm'_\Q \cap \X(B) \with \langle \rho(\Dm), x \rangle \subseteq \Z } \text{.}
  \]
\end{defn}

\begin{cor}
  \label{cor:char-conn}
  A spherical subgroup $H \subseteq G$ is connected if and only if $\Mm$ is \emph{$\Dm$-saturated}.
\end{cor}

\subsection{Organization of the Paper}
\label{sec:orga}
In Section \ref{sec:generalities}, we recall several fundamental results about Luna data. In Section \ref{sec:outlook}, we present the key steps of the proof of our two main theorems. In Section \ref{sec:pullback-col}, we investigate the behavior of a color under $G$-equivariant pullback. These results will be crucial in the proof of Theorem \ref{thm:corrsp-dist-pairs-subgrps} and Theorem \ref{thm:corr-subdata-subgrp} which is stated in Section \ref{sec:finite-quotients}. We complete our work with Section \ref{sec:hsd-id-comp} where we compute the Luna datum of the identity component of a spherical subgroup.

\addtocontents{toc}{\SkipTocEntry}
\section*{Acknowledgments}
The author would like to thank Victor Batyrev for encouragement and highly useful advice, as well as Giuliano Gagliardi for many fruitful discussions and his constant interest in the author's work, as well as for proofreading a preliminary version of the manuscript.

\section{Generalities}
\label{sec:generalities}

\subsection{Luna data}
\label{sec:homog-spher-data}

We recall the combinatorial description of Luna data introduced by Luna \cite{Luna:TypeA} and refer the reader to \cite[Definition~30.21]{Timashev:HSEE} for further details and references.

A spherical subgroup $H \subseteq G$ with Luna datum $( \Mm, \Sigma, S^p, \Dm^a )$ is called \emph{wonderful} if $\Mm = \lspan_\Z \Sigma$. Spherical roots of wonderful subgroups in $G$ of rank $1$, i.e., $\rk(\Mm) = 1$, are said to be \emph{spherical roots of $G$} and were classified by Akhiezer and Brion (see \cite{Akhiezer,Brion:Rank1}). They are nonnegative linear combinations of simple roots with coefficients in $\tfrac{1}{2}\Z$. The  set $\Sigma_G$ of spherical roots of $G$ is finite and can be obtained from Table \ref{tab:sr} (in Remark \ref{rem:howto-sr} we explain how to read Table \ref{tab:sr}).

Let us fix a bit more notation. We write $\X(R)$ for the \emph{root lattice} of a root system $R$ and $\Supp(\gamma) \subseteq S$ for the \emph{support} of $\gamma \in \lspan_\Q (\X(R))$, i.e., the set of simple roots with nonzero coefficient in the expression of $\gamma$ as linear combination of simple roots.

\begin{rem}[How to read Table \ref{tab:sr}]
  \label{rem:howto-sr}
  The second column shows the spherical root $\gamma$ after applying the usual Bourbaki numbering \cite{Bourbaki:Lie} to the simple roots in $\Supp(\gamma)$ whose Dynkin-diagram-type is given in the first column. The simple roots are denoted by $\alpha_i$ if $\Supp(\gamma)$ has one simple factor and $\alpha_1, \alpha_1'$ if $\Supp(\gamma)$ has two simple factors. We have $\lambda \in \{ 1, \tfrac{1}{2}\}$ where $\lambda = 1$ is always possible and $\lambda = \tfrac{1}{2}$ is only possible if the corresponding entry in the second column is in the character lattice $\X(B)$. The third column lists the set of simple roots $S^{pp}(\gamma) \subseteq S$ used in Definition \ref{def:comp-sr}.
\end{rem}

Fix a maximal torus $T \subseteq G$ and a Borel subgroup $B \subseteq G$ containing $T$. Write $R$ for the induced root system and $S$ for the corresponding set of simple roots.
The set of spherical roots $\Sigma_G$ can be straightforwardly extracted from Table \ref{tab:sr} and it splits in two disjoint subsets $(\Sigma_G \cap \X(R))$ and $(\Sigma_G \setminus \X(R))$ depending on whether $\lambda = 1$ or $\lambda = \tfrac{1}{2}$.

\begin{table}[!ht]
  \centering{}
  \begin{tabular}{@{}lll@{}}
    \toprule
    Type of $\Supp(\gamma)$ & $\gamma$ & $S^{pp}(\gamma)$ \\
    \midrule
    $\mathsf{A}_1$ & $\alpha_1$ & $\emptyset$ \\
    $\mathsf{A}_1$ & $2 \alpha_1$ & $\emptyset$ \\
    $\mathsf{A}_1 \times \mathsf{A}_1$ & $\lambda( \alpha_1 + \alpha_1' )$
    & $\emptyset$ \\
    $\mathsf{A}_r$ ($r \ge 2$) & $\alpha_1 + \ldots + \alpha_r$ &
    $\alpha_2, \ldots, \alpha_{r-1}$ \\
    $\mathsf{A}_3$ & $\lambda(\alpha_1 + 2 \alpha_2 + \alpha_3)$ &
    $\alpha_1, \alpha_3$\\
    $\mathsf{B}_r$ ($r \ge 2$) & $\alpha_1 + \ldots + \alpha_r$ &
    $\alpha_2, \ldots, \alpha_{r-1}$ \\
    $\mathsf{B}_r$ ($r \ge 2$) & $2 \alpha_1 + \ldots + 2 \alpha_r$ &
    $\alpha_2, \ldots, \alpha_r$ \\
    $\mathsf{B}_3$ & $\lambda ( \alpha_1 + 2 \alpha_2 + 3 \alpha_3 )$ &
    $\alpha_1, \alpha_2$ \\
    $\mathsf{C}_r$ ($r \ge 3$) & $\alpha_1 + 2 \alpha_2 + \ldots + 2
    \alpha_{r-1} + \alpha_r$ & $\alpha_3, \ldots, \alpha_r$ \\
    $\mathsf{D}_r$ ($r \ge 4$) & $\lambda ( 2 \alpha_1 + \ldots + 2
    \alpha_{r-2} + \alpha_{r-1} + \alpha_r )$ & $\alpha_2, \ldots,
    \alpha_r$ \\
    $\mathsf{F}_4$ & $\alpha_1 + 2 \alpha_2 + 3 \alpha_3 + 2 \alpha_4$ &
    $\alpha_1, \alpha_2, \alpha_3$ \\
    $\mathsf{G}_2$ & $\alpha_1 + \alpha_2$ & $\emptyset$ \\
    $\mathsf{G}_2$ & $2 \alpha_1 + \alpha_2$ & $\alpha_2$ \\
    $\mathsf{G}_2$ & $4 \alpha_1 + 2 \alpha_2$ & $\alpha_2$ \\
    \bottomrule
  \end{tabular}
  \caption{Spherical roots.}
  \label{tab:sr}
\end{table}

\begin{defn}[cf.~{\cite[Section 1.1.6]{BL:F4}} or {\cite[Section 3.4]{Avdeev}}]
  \label{def:comp-sr}
  A pair $(S^p, \gamma)$ with $S^p \subseteq S$ and $\gamma \in \Sigma_G$ is said to be \emph{compatible} if
  \[
    S^{pp}(\gamma) \subseteq S^p \subseteq S^p(\gamma) \text{,}
  \]
  where $S^p(\gamma) \coloneqq \set{ \alpha \in S \with \langle \check{\alpha}, \gamma \rangle = 0 }$ and $S^{pp}(\gamma) \subseteq S$ is given in the third column of Table \ref{tab:sr}.
\end{defn}

\begin{defn}[{\cite{Luna:TypeA}, see also \cite[Definition~30.21]{Timashev:HSEE}}]
  \label{def:hsd}
  A \emph{Luna datum} is a quadruple $(\Mm, \Sigma, S^p, \Dm^a)$ where $\Mm \subseteq \X(B)$ is a sublattice, $\Sigma \subseteq \Mm$ is a linearly independent set of primitive elements, $S^p \subseteq S$, and $\Dm^a$ is a finite set equipped with a map $\rho: \Dm^a \to \Nm \coloneqq \Hom(\Mm, \Z)$ such that the following axioms are satisfied.
  \begin{enumerate}
  \item[($\operatorname{A1}$)] $\langle \rho(D), \gamma \rangle \le 1$ for every $D\in \Dm^a$ and $\gamma \in \Sigma$ with equality holding if and only if $\gamma = \alpha \in \Sigma \cap S$ and $D \in \set{D'_\alpha, D''_\alpha}$ where $D'_\alpha, D''_\alpha \in \Dm^a$ are two distinct elements depending on $\alpha$.
  \item[($\operatorname{A2}$)] $\rho(D'_\alpha) + \rho(D''_\alpha) = \check{\alpha}|_{\Mm}$ for every $\alpha \in \Sigma \cap S$.
  \item[($\operatorname{A3}$)] $\Dm^a = \set{D'_\alpha, D''_\alpha \with \alpha \in \Sigma \cap S}$.
  \item[($\operatorname{\Sigma1}$)] If $\alpha \in \tfrac{1}{2}\Sigma \cap S$, then $\langle \check{\alpha}, \Mm\rangle \subseteq 2\Z$ and $\langle \check{\alpha}, \Sigma \setminus \set{2\alpha}\rangle \le 0$.
  \item[($\operatorname{\Sigma2}$)] If $\alpha, \beta \in S$ are orthogonal and $\alpha + \beta \in \Sigma \cup 2\Sigma$, then $\check{\alpha}|_\Mm = \check{\beta}|_\Mm$.
  \item[($\operatorname{S}$)] $\langle \check{\alpha}, \Mm\rangle = 0$ for every $\alpha \in S^p$ and for every $\gamma \in \Sigma$ the pair $(S^p, \gamma)$ is compatible.
  \end{enumerate}
\end{defn}

\begin{prop}[{\cite[Th\'{e}or\`{e}me 2]{Luna:TypeA}, see also \cite[Theorem 30.22]{Timashev:HSEE}}]
  \label{prop:clc}
  For a spherical subgroup $H \subseteq G$ the quadruple $(\Mm, \Sigma, S^p, \Dm^a)$ as defined above is a Luna datum. This assignment defines a bijection between spherical subgroups $H \subseteq G$ (up to conjugation) and Luna data of $G$.
\end{prop}

\begin{rem}
  Given two Luna data $\Sm_i = ( \Mm_i, \Sigma_i, S^p_i, \Dm^a_i )$ for $i = 1, 2$, all combinatorial objects, except $\Dm^a_i$, are comparable in an obvious way. By writing $\Sm_1 = \Sm_2$, we mean that all these objects coincide and that there is a bijection $\iota \colon \Dm^a_1 \to \Dm^a_2$ such that $\rho_1(D) = \rho_2 \circ \iota( D )$ for every $D \in \Dm^a_1$.
\end{rem}

The full set of colors $\Dm$ may be recovered from a quadruple $(\Mm, \Sigma, S^p, \Dm^a)$ (see, for instance, \cite[Section 2]{GH:HSDose}).

We will also need the following result by Foschi (see also {\cite[Lemma~30.24]{Timashev:HSEE}}).

\begin{prop}[{\cite[Section~2.2, Theorem~2.2]{foschi}}]
  \label{prop:foschi}
  Let $D \in \Dm$ be a color and let $L$ be a $G$-linearized line bundle on $G/H$ with a section $s \in H^0(G/H, L)$ such that $\Div s = D$. Then the section $s$ is $B$-semi-invariant of some weight $\lambda \in \X(B)$, and for $\alpha \in S$ we have
  \[
    \langle \alpha^\vee, \lambda \rangle =
    \begin{cases}
      1 & \text{ if $D \in \Dm(\alpha)$ and $D \in \Dm^a \cup \Dm^b$,} \\
      2 & \text{ if $D \in \Dm(\alpha)$ and $D \in \Dm^{2a}$,}\\
      0 & \text{ if $D \notin \Dm(\alpha)$.}
    \end{cases}
  \]
\end{prop}

\subsection{Luna data of Normalizers}
\label{sec:homog-spher-data-1}

Let $H \subseteq G$ be a spherical subgroup with Luna datum $\Sm = ( \Mm, \Sigma, S^p, \Dm^a )$. We explain how to compute the Luna datum $N_G(\Sm) = ( \Mm^N, \Sigma^N, S^p, \Dm^{a,N} )$ of the normalizer $N \coloneqq N_G(H)$ from $\Sm$. Recall that the (abstract) set $\Dm^{a,N}$ is equipped with two maps $\rho^N \colon \Dm^{a,N} \to \Hom_\Z( \Mm^N, \Z )$ and $\varsigma^N \colon \Dm^{a,N} \to \Pm(S)$.

\begin{defn}[{\cite[Definition 4.1.1]{Losev:Unique}}]
  \label{def:distinguished-sr}
  An element $\gamma \in \Sigma$ is called \emph{distinguished} if
  \begin{enumerate}
  \item $\gamma \in S$ and $\rho(D) = \tfrac{1}{2} \check{\gamma}|_\Mm$ for every $D \in \Dm(\gamma)$ (such colors are called \emph{undetermined}), or
  \item $\gamma \in \X(R) \setminus S$ such that $2\gamma$ is also a spherical root of $G$ compatible with $S^p$.
  \end{enumerate}
  The subset of distinguished roots of $\Sigma$ is denoted by $\Sigma^+$.
\end{defn}

\begin{prop}[{\cite[Theorem 2]{Losev:Unique}}]
  We have
  \[
    \Sigma^N = \rleft( \rleft( \Sigma \cap \X(R) \rright) \setminus \Sigma^+ \rright) \cup \set{ 2 \gamma \with \gamma \in \Sigma^+ \cup \rleft( \Sigma \setminus \X(R) \rright) } \text{,}
  \]
  i.e., a spherical root $\gamma \in \Sigma$ which is either distinguished or not in the root lattice $\X(R)$ gets doubled. All other roots stay the same.
\end{prop}

\begin{prop}[{\cite[Corollary 6.5]{Knop:Auto}}]
  We have $\Mm^N = \lspan_\Z \Sigma^N$. In particular, $\Mm^N$ is contained in $\Mm$.
\end{prop}

Recall that $N$ acts on $G/H$ by translation from the right, i.e., for $n \in N$ and $gH \in G/H$, we have $n \cdot gH = gnH$. The quotient $(G/H)/N$ under this action coincides with $G/N$. By \cite[Lemma 3.1.5]{Losev:Unique}, the natural projection $\pi \colon G/H \twoheadrightarrow G/N$ induces a map $D \mapsto \pi(D)$ from $\Dm$, the colors of $G/H$, onto $\Dm^N$, the colors of $G/N$, and this is the quotient map for the action of $N$ on $\Dm$. Moreover, $\varsigma^N(\pi(D)) = \varsigma(D)$ and $\rho^N(\pi(D)) = \rho(D)|_{\Mm^N}$ for every color $D \in \Dm$. From this observation, the following two statements directly follow.

\begin{prop}
  The stabilizer of the open $B$-orbit in $G/N$ coincides with the stabilizer of the open $B$-orbit in $G/H$, i.e., $S^p$ stays the same.
\end{prop}

\begin{prop}
  We may identify (the abstract set) $\Dm^{a,N}$ with $\bigcup_{\alpha \in \Sigma^N \cap S} \Dm(\alpha)$ such that $\rho^N(D) = \rho(D)|_{\Mm^N}$ for every $D \in \Dm^{a,N}$.
\end{prop}

In summary, we obtain (see also \cite[Section 2.4]{BP:WonderfulSubgrp}).

\begin{prop}
  \label{thm:hsd-normalizer}
  Let $H \subseteq G$ be a spherical subgroup. Then $N_G(H) \subseteq G$ is also spherical with Luna datum $N_G(\Sm) = ( \Mm^N, \Sigma^N, S^p, \Dm^{a,N} )$ where
  \begin{enumerate}
  \item $\Mm^N = \lspan_\Z \Sigma^N$,
  \item $\Sigma^N = \rleft( \rleft( \Sigma \cap \X(R) \rright) \setminus \Sigma^+ \rright) \cup \set{ 2 \gamma \with \gamma \in \Sigma^+ \cup \rleft( \Sigma \setminus \X(R) \rright) }$,
  \item $\Dm^{a,N} = \bigcup_{\alpha \in \Sigma^N \cap S} \Dm(\alpha)$ with $\rho^N(D) = \rho(D)|_{\Mm^N}$ for $D \in \Dm^{a,N}$.
  \end{enumerate}
\end{prop}

We will also need the following observation.
\begin{rem}
  \label{rem:normalizer-finite-incl}
  Let $H, \Ht \subseteq G$ be two spherical subgroups with $H \subseteq \Ht$ such that the quotient $\Ht/H$ is finite. Then $H_0 \coloneqq H^\circ = \Ht^\circ$ and thus, by \cite[Lemma 3.1.1]{Losev:Unique} , $N \coloneqq N_G(H_0) = N_G(H) = N_G(\Ht)$. In particular, we obtain a sequence of inclusions $H_0 \subseteq H \subseteq \Ht \subseteq N$.
\end{rem}

\section{Overview of the Proof of our Main Theorems}
\label{sec:outlook}

Let $H \subseteq G$ be a spherical subgroup. Fix a Borel subgroup $B \subseteq G$ and a maximal torus $T \subseteq B$. We write $R$ for the corresponding root system, $S$ for the induced set of simple roots and $\Sm = ( \Mm, \Sigma, S^p, \Dm^a )$ for the Luna datum of $H$.

Let us first outline the strategy of the proof of Theorem \ref{thm:corrsp-dist-pairs-subgrps}. Our approach to study $\Hm_H$ depends on the following kind of Stein factorization.
\begin{prop}
  \label{prop:divide-prob}
  Let $\Ht \in \Hm_H$. We define $\Ht_0 \coloneqq H \Ht^\circ$ which is also in $\Hm_H$ such that $\Ht_0/H$ is connected and $\Ht/\Ht_0$ is finite.
\end{prop}

In particular, the study of $\Ht \in \Hm_H$ can be split in the two cases: $\Ht_0 \in \Hm_H^{\mathsf{c}}$ and $\Ht \in \Hm_{\Ht_0}^{\mathsf{f}}$ where
\[
  \Hm_{\Ht_0}^f \coloneqq \set{ H' \subseteq G \; \text{closed subgroup such that} \; \Ht_0 \subseteq H' \; \text{and} \; H'/\Ht_0 \; \text{is finite} } \text{.}
\]

The description of $\Hm_H^{\mathsf{c}}$ is in terms of colored subspaces as explained in Section \ref{sec:prelim} (see \cite[Section 4]{Knop:LV}). It remains to describe the set $\Hm_{\Ht_0}^{\mathsf{f}}$. To keep the notation simple let us consider the set $\Hm_H^{\mathsf{f}}$ instead.

\begin{defn}
  \label{def:subdatum-fc}
  A subgroup $\Mmt \subseteq \Mm$ of finite index is called \emph{distinguished}, if $( \Mmt, \emptyset )$ is a distinguished pair. Let us write $\Ss^f \Sm$ (resp.~$\Ss^c \Sm$) for the set of homogeneous subdata coming from distinguished subgroups (resp.~from colored subspaces) (cf.~Definition \ref{def:subdatum}).
\end{defn}

\begin{rem}
  Observe that $\Ss^f\Sm$ can be naturally identified with the set of distinguished subgroups $\Mmt \subseteq \Mm$.
\end{rem}

We obtain the following combinatorial description of $\Hm_H^{\mathsf{f}}$ which will be proved in Section \ref{sec:proof-prop-bfq}.

\begin{thm}
  \label{thm:bijection-finite-quot}
  There is a representative $\Ht$ of $\Smt = ( \Mmt, \Sigmat, \St^p, \Dmt^a ) \in \Ss^f \Sm$ which lies in $\Hm_H^f$ and the assignment $\Smt \mapsto \Ht$ induces an order-reversing bijection $\Ss^f \Sm \to \Hm_H^f$.  
\end{thm}

\begin{rem}
  \label{rem:no-conj-finite}
  Observe that the equivalence relation \enquote{$\sim$}, induced by conjugation, is trivial on $\Hm_H^f$. Indeed, let $\Ht \in \Hm_H^f$ and $g \in G$ with $g \Ht g^{-1} \in \Hm_H^f$. As both $g H^\circ g^{-1}$ and $H^\circ$ are contained in $g \Ht g^{-1}$, it follows that $g H^\circ g^{-1} = H^\circ$, i.e., $g \in N_G(H^\circ)$. By Remark \ref{rem:normalizer-finite-incl}, $N_G(H^\circ) = N_G(\Ht)$, and thus $g \Ht g^{-1} = \Ht$.
\end{rem}

\begin{proof}[Proof of Theorem \ref{thm:corrsp-dist-pairs-subgrps}]
  Observe that any $\Ht \in \Hm_H$ has a unique intermediate spherical subgroup $H \subseteq \Ht_0 \subseteq \Ht$ such that $\Ht_0/H$ is connected and $\Ht/\Ht_0$ is finite, namely $\Ht_0 \coloneqq H \Ht^\circ$. On the other hand any distinguished pair $(\Mmt, \Dm^1)$ has an \enquote{intermediate} colored subspace $(\Mmt^\perp, \Dm^1)$ and  these two objects correspond to each other. The claim straightforwardly follows by the bijection $\{\text{colored subspaces}\} \leftrightarrow \Hm_H^{\mathsf{c}}$ (see \cite[Section 4]{Knop:LV}) and Theorem \ref{thm:bijection-finite-quot}.
\end{proof}
\begin{rem}
  The correspondence between $\Hm_H$ and distinguished pairs extends the two bijections $\Hm_h^c \to \{\text{colored subspaces}\;(\Nm^1, \Dm^1)\}$ and $\Hm_H^f \to \{\text{distinguished subgroups} \;\Mmt \subseteq M\}$.
\end{rem}

\begin{proof}[Proof of Theorem \ref{thm:corr-subdata-subgrp}]
  Let $\Smt  \preceq \Sm$ be a subdatum with corresponding distinguished pair $( \Mmt, \Dm^1 )$ (which need not to be unique). As $(\Mmt^\perp, \Dm^1)$ is a colored subspace of $\Nm_\Q$, there is a spherical subgroup $H_0 \subseteq G$ containing $H$ such that $H_0/H$ is connected (see \cite[Section 4]{Knop:LV}). Let $\Sm_0$ be the Luna datum of $H_0$. It is straightforward to verify that $\Smt \in \Ss^{\mathsf{f}}\Sm_0$, and thus, by Theorem \ref{thm:bijection-finite-quot}, there is a spherical subgroup  $\Ht$ containing $H_0$ such that $\Ht/H_0$ is finite. It follows that the assignment $\Smt \mapsto \Ht$ defines a well-defined injective map $\Ss \Sm \to \Hm_H/\sim$ (see Proposition \ref{prop:clc}).
  
  To show surjectivity let $\Ht \in \Hm_H$ with associated Luna datum $\Smt$. By \cite[Section 4]{Knop:LV}, $\Ht_0 = H \Ht^\circ$ corresponds to a colored subspace $( \Nm^1, \Dm^1 )$. By Theorem \ref{thm:bijection-finite-quot}, $\Smt \in \Ss^{\mathsf{f}}\Smt_0$ where $\Smt_0$ is the Luna datum of $\Ht_0$. It is straightforward to show that $\Smt$ is the subdatum of $\Sm$ corresponding to the distinguished pair $(\Mmt,\Dm^1)$.
\end{proof}

\begin{example}
  Let $G = \Spin_5$. Fix a maximal torus contained in some Borel subgroup, and denote by $\alpha_1, \alpha_2$ the simple roots of $G$ where $\alpha_2$ is the shorter root. According to the $14$th entry in \cite[Table B]{Wasserman}, there exists a wonderful subgroup $H$ of $G$ with Luna datum $\Sm = ( \Mm, \Sigma, S^p, \Dm^a )$ where $\Mm = \Z \alpha_1 \oplus \Z \alpha_2$, $\Sigma = \set{ \alpha_1, \alpha_2 }$, $S^p = \emptyset$ and $\Dm^a = \set{ D_{\alpha_1}^\pm, D_{\alpha_2}^\pm }$ equipped with the map $\rho \colon \Dm^a \to \Nm \coloneqq \Hom( \Mm, \Z )$ satisfying
  \begin{align*}
    \langle \rho(D_{\alpha_1}^+), \alpha_1 \rangle &= 1, & \langle
    \rho(D_{\alpha_1}^+), \alpha_2 \rangle & = 0, &\langle
    \rho(D_{\alpha_1}^-), \alpha_1 \rangle &= 1, & \langle
    \rho(D_{\alpha_1}^-), \alpha_2 \rangle & = -1, \\
    \langle \rho(D_{\alpha_2}^+), \alpha_1 \rangle &= -1, & \langle
    \rho(D_{\alpha_2}^+), \alpha_2 \rangle & = 1, &\langle
    \rho(D_{\alpha_2}^-), \alpha_1 \rangle &= -1, & \langle
    \rho(D_{\alpha_2}^-), \alpha_2 \rangle & = 1 \text{.}
  \end{align*}
  The distinguished pair $( \Z 2\alpha_2, \set{ D_{\alpha_1}^+ } )$ is neither a colored subspace nor a distinguished subgroup.
\end{example}

\section{Equivariant Pullback of a Color}
\label{sec:pullback-col}

In this section, we explain how colors behave under $G$-equivariant pullback, a result which will be needed in the proof of Theorem \ref{thm:bijection-finite-quot}.

Let $\Ht \in \Hm_H$ and consider the natural $G$-equivariant dominant morphism $\psi \colon G/H \to G/\Ht$. We continue to use the notation from Section \ref{sec:generalities} and write $( \Mm, \Sigma, S^p, \Dm^a )$ resp.~$( \Mmt, \Sigmat, \St^p, \Dmt^a )$ for the Luna datum of $H$ resp.~of $\Ht$. Recall from \cite[Section 4]{Knop:LV} that we have an inclusion $\psi^* \colon \Mmt \hookrightarrow \Mm$ and $\psi_*(\Vm) = \Vmt$ where $\psi_* \colon \Nm \to \Nmt$ denotes the map dual to the inclusion $\Mmt \hookrightarrow \Mm$ and $\Vm$ (resp.~$\Vmt$) denotes the valuation cone of $G/H$ (resp.~of $G/\Ht$). Let $\Dm_\psi \subseteq \Dm$ be the colors which $\psi$ maps dominantly to $G/\Ht$. We get a (surjective) map $\psi_* \colon \Dm \setminus \Dm_\psi \to \Dmt; D \mapsto \overline{\psi(D)}$.

In the sequel, we will abbreviate $\Sigma \cap S$ by $\Sigma^a$ and $\Sigma \cap 2S$ by $\Sigma^{2a}$. Recall that we have maps $\rho \colon \Dm \to \Nm$ and $\rhot \colon \Dmt \to \Nmt$.

\begin{prop}
  \label{prop:pullback-col}
  For the pullback $\psi^*(\Dt)$ of $\Dt \in \Dmt$ we have:
  \begin{enumerate}
  \item If $\Dt \in \Dmt^a$, then $\psi^*(\Dt) \in \Dm^a$.
  \item If $\Dt \in \Dmt^b$, then $\psi^*(\Dt) \in \Dm^a \cup \Dm^b$.
  \item If $\Dt \in \Dmt^{2a}$, then either $\psi^*(\Dt) \in \Dm^{2a}$ or $\psi^*(\Dt) = D_1 + D_2$ for $\Dm(\alpha) = \set{ D_1, D_2 }$ where $\varsigma(\Dt) = \set{ \alpha } \subseteq \Sigma^a$.
  \end{enumerate}
\end{prop}

The proof of the three parts of Proposition \ref{prop:pullback-col} will be given in the three subsections \ref{sec:pullback-col-a}, \ref{sec:pullb-col-b} and \ref{sec:pullback-col-2a} respectively.

\subsection{Preliminaries}
\label{sec:prep-obs}

\begin{prop}
  \label{prop:a-roots}
  $\cone( \Sigmat ) = \cone( \Sigma ) \cap \Mmt_\Q$. In particular $\Sigmat^a \subseteq \Sigma^a$.
\end{prop}
\begin{proof}
  We have
  \[
    \cone ( \Sigmat ) = - ( \Vmt )^\vee = - \psi_*( \Vm )^\vee = (
    -\Vm^\vee ) \cap \Mmt_\Q = \cone( \Sigma ) \cap \Mmt_\Q \text{.}
  \]
  As every spherical root is a nonnegative linear combination of simple roots, it follows $\Sigmat^a \subseteq \Sigma^a$.
\end{proof}

Replacing $G$ with a suitable finite covering, we may assume that $G = G^{ss} \times C$ where $G^{ss}$ is semi-simple simply connected and $C$ is a torus. Then, by \cite[Prop. 2.4 and Rem. after it]{KKLV:LPAGA}, any line bundle on $G/H$ is $G$-linearizable. Moreover, as any two linearizations differ by a character of $C$, for any $B$-invariant effective Cartier divisor $\delta$ on $G/H$ there exists a uniquely determined $G$-linearized line bundle $L_\delta$ on $G/H$ together with a $B$-semi-invariant and $C$-invariant global section $s_\delta$ (determined up to proportionality) such that $\Div s_\delta = \delta$.

In particular, every color $D \in \Dm$ defines a $G$-linearized line bundle $L_D$ on $G/H$ together with a $B$-semi-invariant and $C$-invariant global section $s_D$, whose $B$-weight we denote by $\lambda_D \in \X(B)$. Analogously for $\Dt \in \Dmt$ where we use the notation $L_\Dt$, $s_\Dt$ and $\lambda_\Dt$.

\begin{rem}
  \label{rem:pullback-sdt}
  As the pullback $\psi^*( s_\Dt )$ is $B$-semi-invariant and $C$-invariant, there exist $\mu_{\Dt, D} \in \Z_{\ge 0}$ (for $D \in \Dm$) and an isomorphism of $G$-linearized line bundles
  \[
    \psi^* \rleft( L_\Dt \rright) \cong \bigotimes_{D \in \Dm} L_D^{\otimes \mbox{$\mu_{\Dt, D}$}}
  \]
  sending $\psi^* ( s_\Dt )$ to a nonzero multiple of $\bigotimes_{D \in \Dm} s_D^{\otimes \mbox{$\mu_{\Dt, D}$}}$. In particular, $\lambda_\Dt = \sum_{D \in \Dm} \mu_{\Dt, D} \lambda_D$. By Proposition \ref{prop:foschi}, $\varsigma(\Dt)$ coincides with the union of all $\varsigma(D')$ for $D' \in \Dm$ with $\mu_{\Dt,D'} \neq0$.
\end{rem}

\begin{prop}
  \label{prop:Sp}
  $\psi$ induces a map between the open $B$-orbits. In particular $S^p \subseteq \St^p$.
\end{prop}
\begin{proof}
  The map $\psi$ is open (see \cite[Theorem 25.1.2]{TauvelYu:LAAG})) and $G$-equivariant.
\end{proof}

\subsection{Pullback of Colors of Type $a$}
\label{sec:pullback-col-a}

In this section, we prove the first part of Proposition \ref{prop:pullback-col}.

\begin{prop}
  \label{prop:pullback-col-type-a}
  If $\Dt \in \Dmt^a$, then $\psi^*( \Dt ) = D_1 + \ldots + D_k$ for colors $D_1, \ldots, D_k \in \Dm^a$ and the union $\varsigma(\Dt) = \varsigma(D_1) \cup \ldots \varsigma(D_k)$ is disjoint.
\end{prop}
\begin{proof}
  Let $\alpha \in \varsigma( \Dt )$. By Proposition \ref{prop:Sp}, $\Dm(\alpha) \neq \emptyset$ and, by Proposition \ref{prop:foschi},
  \[
    1 = \langle \check{\alpha}, \lambda_\Dt \rangle = \sum_{D \in \Dm} \mu_{\Dt, D} \langle \check{\alpha}, \lambda_D \rangle = \sum_{D \in \Dm(\alpha)} \mu_{\Dt, D} \langle \check{\alpha}, \lambda_D \rangle.
  \]
  By Proposition \ref{prop:a-roots}, $\Dm(\alpha) = \set{ D_{\alpha,1}, \Dm_{\alpha,2} } \subseteq \Dm^a$, so that $\mu_{\Dt, D_{\alpha,1}} = 1$ and $\mu_{\Dt, D_{\alpha,2}} = 0$ or vice versa.

  Going through all simple roots in $\varsigma(\Dt)$, we obtain colors $D_1, \ldots, D_k \in \Dm^a$ such that $\psi^*(\Dt) = D_1 + \ldots + D_k$. If there were $\alpha \in \varsigma(D_i) \cap \varsigma(D_j)$ for $i \neq j$, then, by Proposition \ref{prop:foschi},
  \[
    2 \le \langle \check{\alpha}, \lambda_{D_1} \rangle + \ldots + \langle \check{\alpha}, \lambda_{D_k} \rangle = \langle \check{\alpha}, \lambda_\Dt \rangle
  \]
  which is not possible as $\Dt \in \Dmt^a$.
\end{proof}

\begin{prop}
  \label{prop:rho-of-pullback}
  In the situation of Proposition \ref{prop:pullback-col-type-a}, we have $\rhot(\Dt) = \rho(D_i)|_\Mmt$ for $i = 1, \ldots, k$.
\end{prop}
\begin{proof}
  By Proposition \ref{prop:Sp}, $\psi$ maps the open $B$-orbit $O$ in $G/H$ onto the open $B$-orbit $\widetilde{O}$ in $G/\Ht$. If $f_\chi \in \C( G / \Ht )$ is a $B$-semi-invariant rational function of weight $\chi \in \widetilde{M}$, then $\psi^*(f_\chi) \in \C(G/H)$ is a $B$-semi-invariant rational function of weight $\chi$. If $\langle \rhot( \Dt ), \chi \rangle \ge0$, i.e., $f_\chi$ extends to a regular function on $\widetilde{O} \cup \Dt$, then $\psi^*(f_\chi)$ extends to a regular function on $O \cup D_1 \cup \ldots \cup D_k$, i.e., $\langle \rho( D_i ), \chi \rangle \ge 0$. It follows that $\rhot(\Dt) = l \rho(D_i)|_\Mmt$ for some $l \in \Q_{>0}$. As $\langle \rho(D_i), \alpha \rangle = 1 = \langle \rhot( \Dt ), \alpha \rangle$ for $\alpha \in \varsigma(D_i) \subseteq \varsigma( \Dt )$, we have $l = 1$.
\end{proof}

\begin{cor}
  \label{cor:a-col-come-from-a-col}
  The pullback of Cartier divisors induces an embedding $\psi^* \colon \Dmt^a \hookrightarrow \Dm^a$ which restricts to a bijection $\psi^*|_{\Dmt( \alpha )} \colon \Dmt( \alpha ) \to \Dm( \alpha )$ for each $\alpha \in \Sigmat^a$.
\end{cor}
\begin{proof}
  By Proposition \ref{prop:pullback-col-type-a}, we have $\psi^*(\Dt) = D_1 + \ldots + D_k$ with a disjoint union $\varsigma(\Dt) = \varsigma(D_1) \cup \ldots \cup \varsigma(D_k)$. As $\langle \rhot( \Dt ), \alpha \rangle = 1$ for every $\alpha \in \varsigma(\Dt)$, it follows by Proposition \ref{prop:rho-of-pullback} that $k=1$, i.e., $\psi^*(D) = D \in \Dm^a$ and $\varsigma(\Dt) = \varsigma(D)$. We get an embedding $\psi^* \colon \Dmt^a \hookrightarrow \Dm^a$. Since $\varsigma(\Dt) = \varsigma( \psi^*(\Dt) )$ for every $\Dt \in \Dmt^a$, we obtain bijective restriction $\psi^*|_{\Dmt(\alpha)} \colon \Dmt(\alpha) \to \Dm(\alpha)$ for $\alpha \in \Sigmat^a$.
\end{proof}

\subsection{Intermezzo on Finite Quotients}
\label{sec:interm-finite-quot}

In this section we consider the case that $\Ht/H$ is finite as it plays a crucial role in the remaining cases of Proposition \ref{prop:pullback-col}.

As $\dim G/H = \dim G/\Ht$, the morphism $\psi$ maps no color dominantly onto $G/\Ht$. Moreover, as $\Ht \subseteq N_G(H)$ (see Remark \ref{rem:normalizer-finite-incl}), the finite reductive group $\Gamma = \Ht/H$ acts on $G/H$ by right translations such that we may regard $\psi$ as the quotient morphism $G/H \to ( G / H ) / \Gamma \cong G / \Ht$. Then $\Gamma$ acts on $\Dm$ and the $\Gamma$-orbits in $\Dm$ can be identified with $\Dmt$. In particular, $\psi(D) \in \Dmt$ for any $D \in \Dm$ (cf.~\cite[Theorem 1.2.3.6]{ADHL}).

\begin{prop}
  \label{prop:M-Sigma-Sp-fq}
  We have:
  \begin{enumerate}
  \item $\Mmt$ is a sublattice of $\Mm$ of finite index.
  \item $\cone(\Sigmat) = \cone( \Sigma )$.
  \item $S^p = \St^p$.
  \end{enumerate}
\end{prop}
\begin{proof}
  Follows from \cite[Lemma 3.1.5]{Losev:Unique}.
\end{proof}

\begin{prop}
  \label{prop:type-a-to-type-2a}
  For every $\alpha \in \Sigma^a \setminus \Sigmat^a$, we have $2\alpha \in \Sigmat$ and $\rho(D) = \tfrac{1}{2} \check{\alpha}|_\Mm$ for every $D \in \Dm(\alpha)$.
\end{prop}
\begin{proof}
  As $\Mmt$ is a sublattice of $\Mm$ of finite index and $\cone( \Sigmat ) = \cone( \Sigma )$ (see Proposition \ref{prop:M-Sigma-Sp-fq}), it follows that $2 \alpha \in \Sigmat$ for every $\alpha \in \Sigma^a \setminus \Sigmat^a$.

  Let $\alpha \in \Sigma^a \setminus \Sigmat^a$, $\Dmt(\alpha) = \{ \Dt \}$ ($\Dt$ is of type $2a$) and $\Dm(\alpha) = \set{ D_1, D_2 }$. Recall from Remark \ref{rem:pullback-sdt}, that $\lambda_\Dt = \sum_{D \in \Dm} \mu_{\Dt, D} \lambda_D$ for some $\mu_{\Dt, D} \in \Z_{\ge0}$. By Proposition \ref{prop:foschi}, we obtain
  \[
    2 = \langle \check{\alpha}, \lambda_\Dt \rangle = \sum_{D \in \Dm} \mu_D \langle \check{\alpha}, \lambda_D \rangle = \mu_{D_1} \langle \check{\alpha}, \lambda_{D_1} \rangle + \mu_{D_2} \langle \check{\alpha}, \lambda_{D_2} \rangle = \mu_{D_1} + \mu_{D_2}.
  \]
  As $\psi(D_i)$ for $i = 1, 2$ is a color in $G/\Ht$ which gets moved by $P_\alpha$, it follows that $\mu_{D_1} = \mu_{D_2} = 1$, i.e., $\psi^*(\Dt) = D_1 + D_2$. With the same arguments as in the proof of Proposition \ref{prop:rho-of-pullback}, we obtain $\rho(D_i)|_\Mmt = \tfrac{1}{2} \check{\alpha}|_\Mmt$, and the assertion follows by Proposition \ref{prop:M-Sigma-Sp-fq} (1).
\end{proof}

\begin{prop}
  \label{prop:pullback-col-fq}
  For every $\alpha \in S \setminus S^p$ and every $\Dt \in \Dmt(\alpha)$, we have $\psi^*(\Dt) \in \Dm(\alpha)$ except for $\alpha \in \Sigma^a \setminus \Sigmat^a$, in which case we have $\psi^*(\Dt) = D_1 + D_2$ where $\Dm(\alpha) = \set{ D_1, D_2 }$.
\end{prop}
\begin{proof}
  For $\alpha \in \Sigmat^a$ resp.~$\alpha \in \Sigma^a \setminus \Sigmat^a$, the assertion follows by Corollary \ref{cor:a-col-come-from-a-col} resp.~Proposition \ref{prop:type-a-to-type-2a}.

  If $\alpha \in S \setminus \Sigma \subseteq S \setminus \Sigmat$ (see Proposition \ref{prop:a-roots}), then $| \Dmt(\alpha) | = | \Dm(\alpha) | = 1$, say $\Dt \in \Dmt(\alpha)$ and $D \in \Dm(\alpha)$. We have $\varsigma( \Dt ) = \varsigma(D) = \set{ \alpha }$ unless there exists a spherical root $\gamma \in \Sigma$ with $\Supp(\gamma) = \set{ \alpha, \alpha'}$ for a simple root $\alpha' \in S$ orthogonal to $\alpha$, in which case $\varsigma(D) = \set{ \alpha, \alpha' }$. By Proposition \ref{prop:M-Sigma-Sp-fq}, some multiple of $\gamma$ will be a spherical root of $G/\Ht$ and thus $\varsigma(\Dt) = \set{ \alpha, \alpha' }$. By Proposition \ref{prop:foschi}, we obtain $\lambda_\Dt = \lambda_D$ (see Remark \ref{rem:pullback-sdt}) and the assertion follows.
\end{proof}

\subsection{Pullback of Colors of Type $b$}
\label{sec:pullb-col-b}

We continue to study the (general) $G$-equivariant pullback of a color and drop the assumption $\Ht / H$ finite. With similar arguments as in the proof of Proposition \ref{prop:pullback-col-type-a}, we get the following statement.

\begin{prop}
  \label{prop:pullback-col-type-b}
  If $\Dt \in \Dmt^b$, then $\psi^*( \Dt ) = D_1 + \ldots + D_k$ for colors $D_1, \ldots, D_k \in \Dm^a \cup \Dm^b$ and the union $\varsigma(\Dt) = \varsigma(D_1) \cup \ldots \varsigma(D_k)$ is disjoint.
\end{prop}

\begin{prop}
  \label{prop:b-gives-ab}
  Let $\Dt \in \Dmt^b$. Then $\psi^*(\Dt) \in \Dm^a \cup \Dm^b$.
\end{prop}
\begin{proof}
  If $|\varsigma(\Dt)|=1$, then the assertion follows by Proposition \ref{prop:pullback-col-type-b}. It remains the case $\varsigma(\Dt) = \set{ \alpha, \beta }$ for two orthogonal simple roots $\alpha, \beta \in S$ and $\gamma \coloneqq \lambda(\alpha + \beta) \in \Sigmat$ for $\lambda \in \set{ \tfrac{1}{2}, 1 }$. As any spherical root is a nonnegative linear combination of simple roots, we obtain, by Proposition \ref{prop:a-roots}, either $\alpha, \beta \in \Sigma$ or $\lambda' \gamma \in \Sigma$ for $\lambda' \in \set{ \tfrac{1}{2}, 1 }$ (observe that $2\alpha \in \Sigma$ or $2\beta \in \Sigma$ is not possible by Proposition \ref{prop:foschi}). In the latter case, $\psi^*(\Dt) \in \Dm^b$ and we are left with $\alpha, \beta \in \Sigma$. We distinguish the two possibilities for the coefficient $\lambda$.

  If $\gamma = \tfrac{1}{2} ( \alpha + \beta )$, we may assume $\Dm(\alpha) = \set{ D, D' }$ such that $D$ appear in $\psi^*(\Dt)$. Similarly as in the proof of Proposition \ref{prop:rho-of-pullback}, we obtain $\rhot(\Dt) = l \rho(D)|_\Mmt$ for some $l \in \Q_{>0}$. As $\langle \rhot(\Dt), \gamma \rangle = 1$, it follows that $\langle \rho(D), \alpha + \beta \rangle > 0$. Since $\gamma \in \Mmt \subseteq \Mm$, we have $\langle \rho(D), \gamma \rangle \in \Z$ which implies $\langle \rho(D) , \beta \rangle > 0$. Hence $\varsigma(D) = \set{ \alpha, \beta }$ and $\psi^*(\Dt) = D$.

  Finally, if $\gamma = \alpha + \beta$, we need to exclude the case $\psi^*(\Dt) = D_1 + D_2$ for $D_1, D_2 \in \Dm^a$. By Remark \ref{rem:pullback-sdt} and Proposition \ref{prop:foschi}, we may assume $\varsigma(D_1) = \set{ \alpha }$ and $\varsigma(D_2) = \set{ \beta }$. Again as in the proof of Proposition \ref{prop:rho-of-pullback}, we obtain $\rhot(\Dt) = l_i \rho(D_i)|_\Mmt$ for some $l_i \in \Q_{>0}$. The equation $\langle \rhot(\Dt), \alpha + \beta \rangle = 2$ implies $\langle \rho(D_1), \beta \rangle = \langle \rho(D_2), \alpha \rangle = 0$. Let $\Dm(\alpha) = \set{ D_1, D_1' }$ and $\Dm(\beta) = \set{ D_2, D_2' }$ and observe that $\langle \rho(D_i), \delta \rangle = \langle \rho(D_i'), \delta \rangle$ for $\delta = \alpha, \beta$ and $i = 1, 2$ by axiom $(A2)$. The group $\Ht_0 \coloneqq H\Ht^\circ$ is a spherical subgroup of $G$ lying between $H$ and $\Ht$ such that $\Ht_0/H$ is connected and $\Ht/\Ht_0$ is finite. Let us denote the Luna datum of $\Ht_0$ by $\Sm_0 = ( \Mm_0, \Sigma_0, S^p_0, \Dm^a_0 )$. We have the following commutative diagram of natural morphisms
  \begin{center}
  \begin{tikzcd}
    G/H \ar[r,-latex,"\psi_0"] \ar[rd,-latex,"\psi"']& G/\Ht_0 \ar[d,-latex,"\psit"]\\
    & G/\Ht \text{.}
  \end{tikzcd}
  \end{center}

  We have a natural inclusion $\Mm_0 \hookrightarrow \Mm$ (see \cite[Section 4]{Knop:LV}). By Proposition \ref{prop:M-Sigma-Sp-fq}, $\Mmt \subseteq \Mm_0$ is a subgroup of finite index and $\cone(\Sigma_0) = \cone(\Sigmat)$. Hence $\alpha + \beta \in \Sigma_0$ and $D_0 \coloneqq \psit^*(\Dt) \in \Dm^b_0$ with $\varsigma(D_0) = \set{ \alpha, \beta }$. Let $(\Nm^1_\Q, \Dm^1)$ be the colored subspace associated to  $\Ht_0$ (see \cite[Section 4]{Knop:LV}). As $D_1', D_2'$ do not appear in $\psi^*(\Dt)$, $\psi$ maps $D_1', D_2'$ dominantly onto $G/\Ht$. As $\psit$ maps no color dominantly onto $G/\Ht$, we obtain $D_1', D_2' \in \Dm^1$, but above we have seen $\rho(D_i') \not \in \Nm^1_\Q = \Mm_0^\perp = \Mmt^\perp \subseteq \Nm_\Q$ for $i = 1, 2$. Contradiction.
\end{proof}

\subsection{Pullback of Colors of Type $2a$}
\label{sec:pullback-col-2a}

\begin{prop}
  If $\Dt \in \Dm^{2a}$ with $\varsigma(\Dt) = \set{ \alpha }$, then either $\psi^*(\Dt) \in \Dm^{2a}$ or $\psi^*(\Dt) = D_1 + D_2$ for two colors $D_1, D_2 \in \Dm^a$ with $\varsigma(D_i) = \set{ \alpha }$ for $i = 1, 2$.
\end{prop}
\begin{proof}
  By \cite[Proposition 3]{Serre:Fibres}, the projections $G \to G/H$ and $G \to G/\Ht$ are locally trivial in \'{e}tale topology. As \'{e}tale maps preserve reducedness, it follows that the pullback of Cartier divisors coincides with the preimage of varieties. It follows that the morphism $\psi \colon G/H \to G/\Ht$ also preserves reducedness.

  By Proposition \ref{prop:Sp}, $\Dm(\alpha) \neq \emptyset$ and, by Proposition \ref{prop:foschi}, we obtain
  \[
    2 = \langle \check{\alpha}, \lambda_\Dt \rangle = \sum_{D \in \Dm} \mu_{\Dt, D} \langle \check{\alpha}, \lambda_D \rangle = \sum_{D \in \Dm(\alpha)} \mu_{\Dt, D} \langle \check{\alpha}, \lambda_D \rangle.
  \]
  By Proposition \ref{prop:a-roots}, either $\alpha \in \Sigma$ or $2\alpha \in \Sigma$. If $\Dm(\alpha) = \set{ D_1, D_2 } \subseteq \Dm^a$, we obtain, by the reducedness of $\psi$, $\mu_{\Dt, D_1} = \mu_{\Dt, D_2} = 1$, and thus $\psi^*( \Dt ) = D_1 + D_2$. Otherwise $\Dm(\alpha) = \set{ D }$ with $2\alpha \in \Sigma$ and $\psi^*(\Dt) = D$.
\end{proof}

\section{Finite Quotients}
\label{sec:finite-quotients}

In this section, we prove Theorem \ref{thm:bijection-finite-quot}. Our proof relies on a study of $( B \times H )$-semi-invariant rational functions on $G$ and a generalization of the combinatorial description \cite[Section 6]{Luna:TypeA}.

\subsection{$(B \times H)$-Semi-Invariant Functions}
\label{sec:b-h-semi-inv-func}

Let $H \subseteq G$ be a spherical subgroup, $B \subseteq G$ a Borel subgroup, $T \subseteq B$ a maximal torus and $\Sm = ( \Mm, \Sigma, S^p, \Dm^a )$ the Luna datum of $H$. The identity component $H^\circ$ is a spherical subgroup with Luna datum $\Sm^\circ = ( \Mm^\circ, \Sigma^\circ, S^p, (\Dm^\circ)^a)$ (see Proposition \ref{prop:M-Sigma-Sp-fq}). Recall that $H$ acts on $G / H^\circ$ by right translations such that $\varphi \colon G / H^\circ \to G / H$ is the corresponding quotient morphism which induces a surjective map $D^\circ \mapsto \varphi( D^\circ )$ from the colors $\Dm^\circ$ of $G/H^\circ$ onto the colors $\Dm$ of $G/H$. Moreover, $H$ acts on $\Dm^\circ$ and we may identify $\Dm$ with the set of $H$-orbits in $\Dm^\circ$ such that $\phi \colon \Dm^\circ \to \Dm, D^\circ \mapsto \varphi( D^\circ )$ is the associated quotient map.

\begin{lemma}
  \label{lem:H-orbit-a-col}
  If $\alpha \in ( \Sigma^\circ )^a \setminus \Sigma^a$, then $\Dm^\circ(\alpha) = \set{ D_1^\circ, D_2^\circ }$ is one $H$-orbit and $\varphi( D_1^\circ ) = \varphi( D_2^\circ )$. Further, the assignment which associates to any pair $\Dm^\circ( \alpha )$ the color $\varphi( D_i^\circ ) \in \Dm$ induces an injection
  \[
    \set{ \Dm^\circ( \alpha ) \with \alpha \in ( \Sigma^\circ )^a \setminus \Sigma^a } \hookrightarrow \Dm^{2a}; \Dm^\circ( \alpha ) = \set{ D_1^\circ, D_2^\circ } \mapsto \varphi( D_i^\circ ) \text{.}
  \]
\end{lemma}
\begin{proof}
  As $\varphi$ is $G$-equivariant, $\phi$ induces well-defined maps from $\Dm^\circ(\alpha)$ onto $\Dm(\alpha)$ for every $\alpha \in S \setminus S^p$. The statement follows by Proposition \ref{prop:type-a-to-type-2a}.
\end{proof}

Replacing $H$ with a suitable conjugate, we may assume that $BH$ is open in $G$, i.e., $eH$ is in the open $B$-orbit in $G/H$. Recall that we can assume $G = G^{ss} \times C$ where $G^{ss}$ is semi-simple simply connected and $C$ is a torus in which case $\C[G]$ is factorial (see \cite[Corollary after Proposition 1]{Popov:Pic}). In particular, the pullback of any color $D$ of $G/H$ under the natural projection map $\pi \colon G \to G/H$ is given by an equations $f_D \in \C[G]$ which is unique up to multiplication by an element in $\C[G]^\times$. As the units in $\C[G]$ coincide with scalar multiples of characters of $G$ (see \cite[Proposition 1.2]{KKV:Picard}), there are unique equations $f_D$ with $f_D|_C = 1$. In other words, $f_D$ is invariant under the $C$-action from left and $f_D(e) = 1$. Let $B\times H$ act on $G$ by left and right translations, so that we get an action of  $B\times H$ on $\C[G]$. The equations $f_D$ are $( B \times H )$-semi-invariant. Recall the identification $\X(G) = \X(C)$ and let us write $f_\chi$ if we consider $\chi \in \X(C)$ as an invertible function on $G$.

The following statement generalizes \cite[Lemme 6.2.2]{Luna:TypeA} which  Luna verifies for spherically closed spherical subgroups.

\begin{lemma}
  \label{lem:b-times-h-ev}
  Every $0 \neq f \in \C( G )^{(B \times H )}$ can be uniquely written as
  \[
    f = c f_\chi \prod_{D \in \Dm} f_D^{n_D}
  \]
  where $c \in \C^\times$, $\chi \in \X(C)$ and $(n_D)_{D \in \Dm} \in \Z^\Dm$.
\end{lemma}
\begin{proof}
  Let $0 \neq f \in \C(G)^{(B \times H)}$. If $f(1) = 0$, then $f(BH) = 0$, so that $f = 0$. Hence $f(1) \neq 0$. The restriction $f|_C$ is a $C$-semi-invariant invertible regular function on $C$, hence $f|_C = c \chi$ for some $c \in \C^\times$ and $\chi \in \X(C)$ (see \cite[Proposition 1.2]{KKV:Picard}). It follows that $f = c f_\chi g$ for some $g \in \C(G)^{(B \times H)}$ with $g|_C = 1$. As $\C[G]$ is factorial, every $(B \times H^\circ)$-semi-invariant rational function on $G$ is a a quotient of two semi-invariant regular functions (see \cite[Theorem II.3.3]{PS:AlgGeoiV}) which in turn can be written as a product of irreducible semi-invariant regular functions ($B \times H^\circ$ is connected).

  By \cite[Proposition 3]{Serre:Fibres}, the projection $G \to G/H^\circ$ is locally trivial in \'{e}tale topology, and thus preserves reducedness. Together with the connectedness of $B \times H^\circ$, it follows that the pullback of colors $\Dm^\circ$ under the natural projection map $G \to G/H^\circ$ induces a bijection between $\Dm^\circ$ and the irreducible components of $G \setminus BH = G \setminus BH^\circ$.

  As explained above, choose local equations for the irreducible components of $G \setminus BH$, i.e., every $D^\circ \in \Dm^\circ$ yields an irreducible equation $f_{D^\circ} \in \C[G]$ with $f_{D^\circ}|_C = 1$. Then every irreducible $(B \times H^\circ)$-semi-invariant regular function $g'$ on $G$ with $g'|_C = 1$ coincides with some $f_{D^\circ}$. We obtain
  \[
    f = c f_\chi \prod_{D^\circ \in \Dm^\circ} f_{D^\circ}^{n_{D^\circ}}
  \]
  for $c \in \C^\times$, $\chi \in \X(C)$ and $(n_{D^\circ})_{D^\circ \in \Dm^\circ} \in \Z^{\Dm^\circ}$.

  As the composition of the projections $G \to G / H^\circ \to G / H$ coincides with $G \to G / H$, it follows by Proposition \ref{prop:pullback-col-fq} that we have $f_D = f_{\varphi^*(D)}$ for $D \in \Dm$ unless $D \in \Dm(\alpha)$ for $\alpha \in ( \Sigma^\circ )^a \setminus \Sigma^a$ in which case we have $f_D = f_{D_1^\circ} \cdot f_{D_2^\circ}$ where $\Dm^\circ (\alpha) = \set{ D^\circ_1, D^\circ_2 }$. As $f$ is $H$-semi-invariant, it follows by Lemma \ref{lem:H-orbit-a-col} that $n_{D^\circ_1} = n_{D^\circ_2}$.
\end{proof}

\begin{cor}
  \label{cor:pullback-b-ef}
  If $\chi \in \Mm$ and $f \in \C(G/H)^{(B)}$ of weight $\chi$, then
  \[
    \pi^*f = c f_{-\chi|_C} \prod_{D \in \Dm} f_D^{\langle \rho(D), \chi \rangle}
  \]
  for some $c \in \C^\times$. In particular $\chi = \chi|_C + \sum_{D \in \Dm} \langle \rho(D), \chi \rangle \chi_D$, where $\chi_D$ is the $B$-weight of $f_D$.
\end{cor}
\begin{proof}
  We continue to use the notation from the proof of Lemma \ref{lem:b-times-h-ev}.
  
  As $\pi \colon G \to G/H$ is $G$-equivariant, $\pi^* f$ is in $( \C(G)^H )^{(B)}$, and thus $(B \times H)$-semi-invariant. By Lemma \ref{lem:b-times-h-ev}, there exist unique $c \in \C^\times$, $\gamma \in \X(C)$ and $(n_D)_{D \in \Dm} \in \Z^\Dm$ such that
  \[
    \pi^*f = c f_\gamma \prod_{D \in \Dm} f_D^{n_D} \text{.}
  \]
  As the $B$-weights of $\pi^*f$ and $f$ coincide, $(\pi^*f)|_C$ is a scalar multiple of $f_{-\chi|_C}|_C$, and thus $\gamma = -\chi|_C$.

  By \cite[Propoosition 3]{Serre:Fibres}, the natural projection map $\pi \colon G \to G/H$ is a locally trivial fibration in \'{e}tale topology. Hence $\nu_{D^\circ}( \pi^*f ) = \nu_{\pi(D^\circ)}(f)$ for every $D^\circ \in \Dm^\circ$, where $D^\circ$ is identified with its corresponding irreducible component in $G \setminus BH$ and $D \coloneqq \pi(D^\circ) \in \Dm$ (see Section \ref{sec:interm-finite-quot}). We have $\nu_D(f) = \langle \rho(D), \chi \rangle$ and $\nu_{D^\circ}( c f_\gamma \prod_{D' \in \Dm} f_{D'}^{n_{D'}} ) = \nu_{D^\circ}( f_D^{n_D} ) = n_D$ (see the last paragraph of the proof of Lemma \ref{lem:b-times-h-ev}). The statement follows.
\end{proof}

\subsection{A Preliminary Correspondence}
\label{sec:prel-corr}

In this section, let $\Ht \subseteq G$ be a spherical subgroup with Luna datum $\Smt = ( \Mmt, \Sigmat, \St^p, \Dmt^a )$. By Lemma \ref{lem:b-times-h-ev}, we have an isomorphism of abelian groups
\[
  \theta \colon \X( C ) \times \Z^\Dmt \to \C(G)^{( B \times \Ht )} / \C^\times; \rleft( \chi, \rleft( n_\Dt \rright)_{\Dt \in \Dmt} \rright) \mapsto f_\chi \prod_{\Dt \in \Dmt} f_\Dt^{ n_\Dt } \text{,}
\]
which gives rise to the following homomorphism of abelian groups:
\[
  \tau \colon \X( C ) \times \Z^\Dmt \to \X( \Ht ); \; x \mapsto \Ht \text{-weight of} \; \theta(x) \text{.}
\]

\begin{rem}
  \label{rem:ker-tau}
  If $\varthetat \colon \Mmt \to \X(C) \times \Z^\Dmt; \chi \mapsto ( -\chi|_C, ( \langle \rhot( \Dt ), \chi \rangle )_{\Dt \in \Dmt} )$, then, by Lemma \ref{lem:b-times-h-ev}, $\widetilde{\Phi} \coloneqq \ker(\tau) = \im(\varthetat)$. Furthermore $\varthetat$ is injective (see \cite[Section 6.3]{Luna:TypeA}).
\end{rem}

Define $\Ht^* \coloneqq \bigcap_{\chi \in \X( \Ht )} \ker( \chi )$ which is the smallest closed normal subgroup of $\Ht$ such that $\Ht/\Ht^*$ is diagonalizable, and we have $\X( \Ht ) = \X( \Ht / \Ht^* )$ (see, for instance, \cite[Exercise 16.12]{Humphreys:LAG}). For every intermediate subgroup $H$ between $\Ht^*$ and $\Ht$, the character group $\X( \Ht / H )$ can be identified with the subgroup of characters in $\X( \Ht )$ whose kernel contain $H$.

Luna proves the following statement (cf.~\cite[Lemme 6.3.1]{Luna:TypeA}) for spherically closed spherical subgroups $\Ht$ which we generalize to arbitrary spherical subgroups.
\begin{lemma}
  \label{lem:intermediate-grps}
  The map $H \mapsto \Phi \coloneqq \tau^{-1}( \X( \Ht / H ) )$ is an order-reversing bijection between the set of subgroups $H \subseteq \Ht$ containing $\Ht^*$ and the set of subgroups $\Phi \subseteq \X(C) \times \Z^\Dmt$ containing $\widetilde{\Phi}$.
\end{lemma}
\begin{proof}
  Follows from the same arguments as in the proof of \cite[Lemme 6.3.1]{Luna:TypeA} with the exception that instead of \cite[Lemme 6.2.2]{Luna:TypeA}, we need to apply Lemma \ref{lem:b-times-h-ev}.
\end{proof}

\subsection{Proof of Theorem \ref{thm:bijection-finite-quot}}
\label{sec:proof-prop-bfq}

We continue to use the notation and assumptions from Section \ref{sec:outlook}.

By replacing $H$ with a suitable conjugate, we may assume that $BH$ is open in $G$. Set $N \coloneqq N_G(H)$ which is a spherical subgroup with Luna datum $N_G(\Sm) = ( \Mm^N, \Sigma^N, S^p, \Dm^{a,N})$ (see Theorem \ref{thm:hsd-normalizer}). By \cite[Proposition 5.1 and Corollaire 5.2]{BP:Val}, $BN = BH$, and thus $BN \subseteq G$ is open. Recall that we have unique equations $f_{D^N} \in \C[G]$ of the pullback of $D^N \in \Dm^N$ under $G \to G/N$ and similarly $f_D \in \C[G]$ for the pullback of $D \in \Dm$ under $G \to G/H$ (see Section \ref{sec:b-h-semi-inv-func}). The action of $N$ on $G/H$ by right translations induces an action of $N$ on $\Dm$, so that $\Dm^N$ can be identified with the $N$-orbits in $\Dm$. Let $\varphi \colon G/H \to G/N$ be the quotient map for the $N$-action on $G/H$. By Section \ref{sec:pullback-col}, the pullback of a color $D^N \in \Dm^N$ under $\varphi$ is a color of $G/H$ unless $D^N \in \Dm^N(\alpha)$ for $\alpha \in \Sigma^a \setminus ( \Sigma^N )^a$, in which case $\Dm(\alpha) = \set{ D_1, D_2 }$ and $\varphi^*(D^N) = D_1 + D_2$. In the first case, we have $f_{D^N} = f_D$ while in the latter $f_{D^N} = f_{D_1} f_{D_2}$. Recall from Theorem \ref{thm:hsd-normalizer} that $\alpha \in \Sigma^a \setminus ( \Sigma^N )^a$ if and only if $\rho(D) = \tfrac{1}{2} \check{\alpha}|_\Mm$ for every $D \in \Dm(\alpha)$.

We define $N^* \coloneqq \bigcap_{\chi \in \X(N)} \ker(\chi)$ and write $\tau^N \colon \X( C ) \times \Z^\Dmt \to \X( N )$ for the map from Section \ref{sec:prel-corr}. As $N/H$ is diagonalizable (see \cite[Theorem 6.1]{Knop:LV} ), $H$ contains $N^*$ and thus, by Lemma \ref{lem:intermediate-grps}, corresponds to a sublattice $\Phi \subseteq \X(C) \times \Z^{\Dm^N}$ containing $\Phi^N \coloneqq \ker(\tau^N)$.

Let $\Ht$ be a spherical subgroup corresponding to $\Smt$ and let $\Dmt$ its full set of colors. By Theorem \ref{thm:hsd-normalizer}, $N_G(\Smt) = N_G(\Sm)$, so that, after replacing $\Ht$ with a suitable conjugate, we may assume $N = N_G(\Ht)$. Pullback of divisors under the natural maps $G/H \to G/N$ resp.~$G/\Ht \to G/N$ induces maps $\iota \colon \Dm^N \to \Dm$ resp.~$\iotat \colon \Dm^N \to \Dmt$: By Section \ref{sec:pullback-col}, the pullback of a color is either a color or a sum of two colors, and in the latter case let any of the two summands be the image of $\iota$ resp.~$\iotat$. Although the maps $\iota$ resp.~$\iotat$ are not unique in general, their compositions with $\rho$ resp.~$\rhot$ are unique, so that we obtain well-defined maps:
\begin{align*}
  \vartheta \colon \Mm & \to \X(C) \times \Z^{\Dm^N}; \chi \mapsto (
  -\chi|_C, ( \langle \rho \circ \iota(D), \chi \rangle )_{D \in
    \Dm^N}) \text{,} \\
  \varthetat \colon \Mmt & \to \X(C) \times \Z^{\Dm^N}; \chi \mapsto
  (-\chi|_C, ( \langle \rhot \circ \iotat(D), \chi \rangle )_{D \in
    \Dm^N}) \text{.}
\end{align*}

By Corollary \ref{cor:pullback-b-ef}, for $f \in \C(G/H)^{(B)}$ of weight $\chi \in \Mm$, we have that $\pi^* f$ and $\theta \circ \vartheta( \chi )$ coincide up to a nonzero scalar multiple where $\pi \colon G \to G/H$, and thus $\vartheta(\Mm) = \Phi$. Similarly $\Phit = \varthetat(\Mmt)$ corresponds to $\Ht$ under the bijection of Lemma \ref{lem:intermediate-grps}. As $\Mmt \subseteq \Mm$ and $\vartheta$ extends $\varthetat$, we obtain $\Phit \subseteq \Phi$, and thus $H \subseteq \Ht$. By Lemma \ref{lem:intermediate-grps}, $N / \Ht$ (resp.~$N / H$) is a quasitorus with character lattice $\Phit / \Phi^N$ (resp.~$\Phi / \Phi^N$), and thus $\dim \Ht - \dim H = \rk( \Phi / \Phit )$. As $\vartheta$ is injective (see Remark \ref{rem:ker-tau}), we obtain $\Phi \cong \Mm$ and $\Phit \cong \Mmt$, and hence $\dim \Ht - \dim H = \rk( \Mm / \Mmt ) = 0$. 

We have seen that the assignment $\Smt \mapsto \Ht$ induces a well-defined and injective map $\Ss^f \Sm \to \Hm^f_H$. Surjectivity follows by Section \ref{sec:pullback-col}.

\section{The Luna datum of the Identity Component}
\label{sec:hsd-id-comp}

In this section, let $H \subseteq G$ be a spherical subgroup with Luna datum $\Sm = (\Mm, \Sigma, S^p, \Dm^a)$. Let us write $\Dm$ for the full set of colors of $\Sm$. The identity component $H^\circ$ of $H$ is also spherical in $G$ and we write $\Sm^\circ = ( \Mm^\circ, \Sigma^\circ, S^p, (\Dm^\circ)^a )$ for its Luna datum (recall that the stabilizer of the open $B$-orbit stays the same, by Proposition \ref{prop:M-Sigma-Sp-fq}, as $H / H^\circ$ is finite). Let us determine $\Sm^\circ$ in terms of $\Sm$.

\begin{prop}
  \label{prop:sph-subgrp-below}
  For any Luna datum $\Smt = ( \Mmt, \Sigmat, \St^p, \Dmt^a )$ with $\Sm \in \Ss^{\mathsf{f}} \Smt$ there is a corresponding spherical subgroup $\Ht \subseteq G$ with $\Ht \subseteq H$ and $H / \Ht$ is finite, and this assignment induces an order-reversing bijection
  \[
    \rleft\{\begin{varwidth}{5in}{spherical subgroups $\Ht \subseteq H
        \subseteq G$ \\ with finite quotient $H / \Ht$}
    \end{varwidth}\rright\}
    \leftrightarrow \rleft\{\begin{varwidth}{5in}{homogeneous
        spherical data \\ $\Smt$ with $\Sm \in \Ss^{\mathsf{f}}
        \Smt$}\end{varwidth}\rright\}\text{.}
  \]
\end{prop}
\begin{proof}
  By the existence part of the Luna conjecture, there exists a spherical subgroup $\Ht \subseteq G$ with corresponding Luna datum $\Smt$. As $\Sm \in \Ss^{\mathsf{f}} \Smt$, it follows, by Theorem \ref{thm:bijection-finite-quot}, that (up to conjugation) $\Ht \subseteq H$ and $H/\Ht$ is finite. Recall from Remark \ref{rem:no-conj-finite} that $\Ht$ is uniquely determined. It follows that the arrow \enquote{$\leftarrow$} is well-defined and injective. By Theorem \ref{thm:bijection-finite-quot}, the arrow \enquote{$\leftarrow$} is also surjective.
\end{proof}

We construct a Luna datum $\Sm^\circ=(\Mm^\circ, \Sigma^\circ, S^p, ( \Dm^\circ )^a)$:
\[
  \Mm^\circ \coloneqq \set{ x \in \Mm_\Q \cap \X(B) \with \langle \rho( \Dm ), x \rangle \subseteq \Z } \text{.}
\]
Then $\Mm$ is a sublattice of $\Mm^\circ$ of finite index. We let $\Sigma^\circ$ be the set of primitive generators in $\Mm^\circ$ of the extremal rays of $\cone( \Sigma ) \subseteq \Mm^\circ_\Q = \Mm_\Q$.

We need the following observation on the spherical roots of $\Sm^\circ$ of type $a$.
\begin{prop}
  \label{prop:id-comp-sph-root-type-a}
  We have $\rleft( \Sigma^\circ \rright)^a = \Sigma^a \cup \rleft( \tfrac{1}{2} \Sigma^{2a} \cap \Mm^\circ \rright)$.
\end{prop}
\begin{proof}
  \enquote{$\subseteq$}: Let $\gamma = \alpha \in ( \Sigma^\circ )^a = \Sigma^\circ \cap S$. Then $k\alpha \in \Sigma$ for some positive integer $k \in \Z$. An inspection of Table \ref{tab:sr} reveals that $k=1$ or $k=2$.

  \enquote{$\supseteq$}: If $\gamma = \alpha \in \Sigma^a$. Then $\Dm( \alpha ) = \set{ D_\alpha^\pm }$ with $\langle \rho(D_\alpha^\pm), \alpha \rangle = 1$. In particular, $\alpha \in \Mm^\circ$ is primitive, so that $\alpha \in ( \Sigma^\circ )^a$. Next consider the case $\gamma = 2\alpha \in \Sigma^{2a} = \Sigma \cap 2S$ with $\alpha \in \Mm^\circ$. Let $D$ be the unique color in $\Dm(\alpha)$. Then $\langle \rho(D), \alpha \rangle = \langle \tfrac{1}{2} \check{\alpha}, \alpha \rangle = 1$, and thus $\alpha \in \Mm^\circ$ is primitive.
\end{proof}

Finally we define the abstract set $(\Dm^\circ)^a$ equipped with a map $\rho^\circ \colon (\Dm^\circ)^a \to \Nm^\circ \coloneqq \Hom( \Mm^\circ, \Z )$. By Proposition \ref{prop:id-comp-sph-root-type-a}, a spherical root of $\Sm^\circ$ of type $a$ can either come from spherical roots of $\Sm$ of type $a$ or from spherical roots of $\Sm$ of type $2a$. For every $\alpha \in \rleft( \tfrac{1}{2} \Sigma \rright) \cap S \cap \Sigma^\circ$, we introduce pairwise distinct elements $D_\alpha^{\pm}$ and set $\rho^\circ( D_\alpha^{\pm} ) \coloneqq \tfrac{1}{2} \check{\alpha}|_{\Mm^\circ}$. Moreover for $D \in \Dm^a$, we set $\rho^\circ(D) \coloneqq \rho(D)|_{\Mm^\circ}$ which is possible as $\Mm$ is a sublattice of $\Mm^\circ$ of finite index. Then $(\Dm^\circ)^a$ is the disjoint union of $\Dm^a$ and the elements $D_\alpha^\pm$ for $\alpha \in \rleft( \tfrac{1}{2} \Sigma \rright) \cap S \cap \Sigma^\circ$.

\begin{prop}
  $\Sm^\circ \coloneqq ( \Mm^\circ, \Sigma^\circ, S^p, ( \Dm^\circ )^a )$ is a Luna datum.
\end{prop}
\begin{proof}
  To show that $( \Mmt, \Sigma^\circ, S^p, ( \Dm^\circ )^a )$ is a Luna datum we verify the axioms of Definition \ref{def:hsd}. The axioms ($\operatorname{A2}$), ($\operatorname{A3}$), and ($\operatorname{\Sigma2}$) straightforwardly follow. We check the remaining axioms.

  We start with axiom ($\operatorname{S}$). The first part of it is straightforward. Let $\gamma \in \Sigma^\circ$ such that $k \gamma \in \Sigma$ for some integer $k \in \Z_{\ge1}$. We have to show that $(S^p, \gamma)$ is compatible. If $\gamma \in \Sigma$, then there is nothing to show. So let us assume $k>1$. An inspection of Table \ref{tab:sr} yields that any spherical root $\gamma'$ comes with a color $D \in \Dm$ such that $\langle \rho(D), \gamma' \rangle \in \set{ 1, 2 }$. Hence $\langle \rho(D), k\gamma \rangle = 2$ and $k = 2$. Another inspection of Table \ref{tab:sr} yields that $(S^p, \gamma)$ is compatible.

  Next we verify axiom ($\operatorname{A1}$). Let $D^\circ \in ( \Dm^\circ )^a$ and $\gamma \in \Sigma^\circ$.  We distinguish two cases.
  \begin{description}
  \item[$D^\circ \in \Dm^a$] If $k\gamma \in \Sigma$ for some integer $k>1$, then $\langle \rho(D^\circ), k\gamma \rangle \neq 1$, and thus $\langle \rho^\circ(D^\circ), \gamma \rangle \le 0$. If $\gamma \in \Sigma$, then $\langle \rho^\circ(D^\circ), \gamma \rangle = \langle \rho(D^\circ), \gamma \rangle \le 1$. We have $\langle \rho^\circ(D^\circ), \gamma \rangle = 1$ if and only if $\gamma = \alpha \in \Sigma \cap S \subseteq \Sigma^\circ \cap S$ and $D^\circ \in \Dm( \alpha ) = \Dm^\circ( \alpha )$.
  \item[$D^\circ \in (\Dm^\circ)^a \setminus \Dm^a$] There exists $\alpha \in \rleft( \tfrac{1}{2} \Sigma \rright) \cap S \cap \Sigma^\circ$ such that $D^\circ \in \{ D_\alpha^ \pm \}$ and $\rho^\circ(D^\circ) = \tfrac{1}{2} \check{\alpha}|_{\Mm^\circ}$. By axiom ($\operatorname{\Sigma1}$) for the Luna datum $\Sm$, we have $\langle \tfrac{1}{2} \check{\alpha}, \Sigma \setminus\set{ 2\alpha } \rangle \le 0$. Hence for every spherical root $\gamma \in \Sigma^\circ \setminus \set{ \alpha }$, we have $\langle \rho^\circ(D^\circ), \gamma \rangle \le 0$. Moreover $\langle \rho^\circ(D^\circ), \alpha \rangle = \langle \rho^\circ(D_\alpha^\pm), \alpha \rangle = 1$.
  \end{description}

  Finally we check axiom $(\Sigma 1)$. Let $\alpha \in \tfrac{1}{2}\Sigma^\circ \cap S$. As $\tfrac{1}{2}\Sigma^\circ \cap S \subseteq \tfrac{1}{2}\Sigma \cap S$, we have $\langle \check{\alpha}, \Sigma \setminus \set{ 2 \alpha } \rangle \le 0$ from which it straightforwardly follows that $\langle \check{\alpha}, \Sigma^\circ \setminus \set{ 2 \alpha } \rangle \le 0$. Let $D \in \Dm$ be the unique color moved by $\alpha$. By construction, $\langle \tfrac{1}{2} \check{\alpha}, \Mm^\circ \rangle = \langle \rho(D), \Mm^\circ \rangle \subseteq \Z$.
\end{proof}

\begin{cor}[of Proposition \ref{prop:sph-subgrp-below}]
  \label{cor:hsd-id-comp}
  The Luna datum of $H^\circ$ is given by $\Sm^\circ$.
\end{cor}
\begin{proof}
  Follows by the bijection in Proposition \ref{prop:sph-subgrp-below}, as $H^\circ$ is the unique minimal element on the left hand side while $\Sm^\circ$ is the unique maximal element on the right hand side.
\end{proof}

\subsection{Examples}
\label{sec:ex}

The following examples illustrate several phenomena of the Luna datum $\Sm^\circ = ( \Mm^\circ, \Sigma^\circ, S^p, (\Dm^\circ)^a )$. In Example \ref{ex:id-comp-1}, we consider a connected spherically closed spherical subgroup such that there is a color of type $a$ preventing a spherical root to be replaced by its half. In Example \ref{ex:id-comp-2}, we investigate a disconnected spherically closed spherical subgroup such that there is a color of type $2a$ preventing a spherical root to be replaced by its half. In Example \ref{ex:id-comp-g2}, we look at a connected spherically closed spherical subgroup such that there are two colors of type $2a$ preventing each other to be replaced by two colors of type $a$ respectively. In Example \ref{ex:id-comp-latt}, we have a spherically closed spherical subgroup which is not connected and the spherical roots of the subgroup and its identity component coincide. Finally, in Example \ref{ex:conn-depends-on-grp}, we have a connected spherical subgroup of a semi-simple group which is not simply connected, and taking the preimage under the isogeny of the universal cover yields a disconnected spherical subgroup. We can detect all these properties combinatorially.

\begin{example}
  \label{ex:id-comp-1}
  The following example is taken from \cite[Table B]{Wasserman}. Let $G = \Spin_7$ and $H$ the subgroup of $G$ isomorphic to $\C^\times \times \Spin_5$. Choose a Borel subgroup $B \subseteq G$ and a maximal torus $T \subseteq B$. We use the usual Bourbaki numbering of the induced set of simple roots $S = \set{ \alpha_1, \alpha_2, \alpha_3 }$. By Wasserman, we have $\Sigma = \set{ \alpha_1, 2 \alpha_2 + 2 \alpha_3 }$, $\Mm = \lspan_\Z \Sigma$, $S^p = \set{ \alpha_3 }$ and $\Dm^a = \set{ D^\pm }$ with $\langle \rho( D^\pm ), \alpha_1 \rangle = 1$ and $\langle \rho( D^\pm ), 2 \alpha_2 + 2 \alpha_3 \rangle = -1$. Let $\gamma \coloneqq 2 \alpha_2 + 2 \alpha_3$. According to Table \ref{tab:sr}, $\gamma' \coloneqq \tfrac{1}{2} \gamma = \alpha_2 + \alpha_3$ is also a spherical root of $G$. It is straightforward to verify that $S^{pp}(\gamma') = \emptyset$ and $S^p(\gamma') = \set{ \alpha_3 }$, so that $( S^p, \gamma' )$ is compatible. Nevertheless $\gamma' \not \in \Sigma^\circ$ as $\langle \rho( D^\pm ), \gamma' \rangle = -\tfrac{1}{2} \not \in \Z$. Indeed we already have $H = H^\circ$.
\end{example}

\begin{example}
  \label{ex:id-comp-2}
  Consider $G = \Spin_7$ and $H' = N_G( H )$ the normalizer of the subgroup form Example \ref{ex:id-comp-1}. We use the notation introduced there. By Example \ref{ex:id-comp-1} and Theorem \ref{thm:hsd-normalizer}, we have $\Sigma = \set{ 2 \alpha_1, 2 \alpha_2 + 2 \alpha_3 }$, $\Mm = \lspan_\Z \Sigma$, $S^p = \set{ \alpha_3 }$ and $\Dm^a = \emptyset$. As in Example \ref{ex:id-comp-1}, we have that $( S^p, \alpha_2 + \alpha_3 )$ is compatible but still $\alpha_2 + \alpha_3 \not \in \Sigma^\circ$ because $\langle \rho^\circ( D ), \alpha_2 + \alpha_3 \rangle = \langle \tfrac{1}{2} \check{\alpha}_1, \alpha_2 + \alpha_3 \rangle = - \tfrac{1}{2} \not \in \Z$ where $D$ is the unique color of type $2a$ moved by $P_{\alpha_1}$. However $\alpha_1 \in \Sigma^\circ$, so that $\Sigma^\circ = \set{ \alpha_1, 2 \alpha_2 + 2 \alpha_3 }$. Indeed, it is straightforward to verify that $( H' )^\circ \cong \C^\times \times \Spin_5$.
\end{example}

\begin{example}
  \label{ex:id-comp-g2}
  The following example is taken from \cite[Table G]{Wasserman}. Let $G = G_2$ and $H \subseteq G$ a subgroup isomorphic to $\SL_2 \times \SL_2$. Choose a Borel subgroup $B \subseteq G$ and a maximal torus $T \subseteq B$. We use the usual Bourbaki numbering of the induced set of simple roots $S = \set{ \alpha_1, \alpha_2 }$. By Wasserman, we have $\Sigma = \set{ 2 \alpha_1, 2 \alpha_2 }$, $\Mm = \lspan_\Z \Sigma$, $S^p = \emptyset$ and $\Dm^a = \emptyset$. Let $\Dm(\alpha_i) = \set{ D_i }$ for $i = 1, 2$. We have $\langle \rho(D_1), 2 \alpha_2 \rangle = -3$ and $\langle \rho(D_2), 2 \alpha_1 \rangle = -1$. Let $\gamma \coloneqq 2 \alpha_i$ for $i = 1, 2$. According to Table \ref{tab:sr}, $\gamma' \coloneqq \tfrac{1}{2} \gamma = \alpha_i$ is a spherical root of $G$. Moreover, it is straightforward to verify that $S^{pp}(\gamma') = \emptyset$ and $S^p(\gamma') = \emptyset$, such that $( S^p, \gamma' )$ is compatible. But still $\gamma' \not \in \Sigma^\circ$, as
  \[
    \langle \rho( D_j ), \gamma' \rangle =
    \begin{cases}
      -\tfrac{1}{2} & \text{for} \; i = 1, \; j = 2 \text{,} \\
      -\tfrac{3}{2} & \text{for} \; i = 2, \; j = 2 \text{.}
    \end{cases}
  \]
  Hence $\Sigma^\circ = \Sigma$. It is straightforward to verify that $\Mm^\circ = \Mm$, and thus $H$ is already connected.
\end{example}

\begin{example}
  \label{ex:id-comp-latt}
  The following example is taken from \cite[Table A]{Wasserman}. Let $G = \SL_2 \times \SL_2$ and $H = BC_G \subseteq G$ the wonderful subgroup which can be written as the product of a Borel subgroup $B \subseteq \SL_2$ diagonally embedded into $G$ and the center $C_G$ of $G$. Choose a Borel subgroup $B' \subseteq G$ and a maximal torus $T \subseteq B'$. We write $S = \set{ \alpha, \alpha' }$ for the induced set of simple roots. By Wasserman, we have $\Sigma = \set{ \alpha, \alpha' }$, $\Mm = \lspan_\Z \Sigma$, $S^p = \emptyset$ and $\Dm^a = \set{ D^+_{\alpha, \alpha'}, D^-_\alpha, D^-_{\alpha'} }$ equipped with the map $\rho \colon \Dm^a \to \Nm \coloneqq \Hom_\Z( \Mm, \Z )$ given by
  \[
    \rho(D^+_{\alpha, \alpha'}) = ( 1, 1 ), \quad \rho(D^-_\alpha) = ( 1, -1 ), \quad \rho(D^-_{\alpha'}) = ( -1, 1)
  \]
  with respect to the basis of $\Nm$ dual to the basis $\alpha, \alpha'$ of $\Mm$. Observe that $H$ is not connected, namely $H^\circ = B$. By Table \ref{tab:sr}, it is clear that $\Sigma^\circ = \Sigma$. So the lattice $\Mm^\circ$ has to be different from $\Mm$. A straightforward computation shows that $\Mm^\circ = \Z \tfrac{1}{2}( \alpha + \alpha' ) \oplus \Z \alpha$.
\end{example}

\begin{example}
  \label{ex:conn-depends-on-grp}
  Let $G = \PGL_2 \times \PGL_2$ and $H = \PGL_2$ where $\PGL_2$ gets diagonally embedded into $G$. Fix a Borel subgroup $B$ and a maximal torus $T$ such that the corresponding root system is given by $R = \set{ \pm \alpha, \pm \alpha' }$ with simple roots $S = \set{ \alpha, \alpha' }$. Recall that $\X(B) = \Z \alpha \oplus \Z \alpha'$. It is straightforward to verify that the Luna datum $\Sm$ of $G/H$ is given by
  \[
    \Mm = \Z ( \alpha + \alpha' ), \Sigma = \set{ \alpha + \alpha' }, S^p = \emptyset, \Dm^a = \emptyset \text{.}
  \]
  Thus, by Corollary \ref{cor:char-conn}, $H \subseteq G$ is connected. Let $\Gt = \SL_2 \times \SL_2$ and $\Ht =N_\Gt( \SL_2 )$ where $\SL_2$ gets diagonally embedded into $\SL_2 \times \SL_2$. Observe that there is an isogeny $\pi \colon \Gt \to G$ such that $\pi^{-1}(H) = \Ht$ and $\Gt$ is the universal cover of $G$. Then $G / H \cong \Gt / \Ht$. Choose a Borel subgroup $\Bt \subseteq \Gt$ and a maximal torus $\Tt \subseteq \Bt$ such that $\pi(\Bt) = B$ and $\pi(\Tt) = T$. The root systems and simple roots of $\Gt$ and $G$ get identified via the comorphism $\pi^* \colon \X(B) \to \X(\Bt)$. Moreover, $\pi^*$ also identifies the Luna datum of $\Gt/ \Ht$ with $\Sm$. Recall that $\X(\Bt) = \Z \tfrac{\alpha}{2} \oplus \Z \tfrac{\alpha'}{2}$ such that, by Corollary \ref{cor:char-conn}, $\Ht$ is not connected, as
  \[
    \set{ x \in \Mm_\Q \cap \X\rleft( \Bt \rright) \colon \langle \rho( \Dm ), x \rangle \subseteq \Z } = \Z \tfrac{ \alpha + \alpha' }{2} \text{.}
  \]
  Indeed $\Ht^\circ = \SL_2$.
\end{example}

 \newcommand{\noop}[1]{}
\providecommand{\bysame}{\leavevmode\hbox to3em{\hrulefill}\thinspace}
\providecommand{\MR}{\relax\ifhmode\unskip\space\fi MR }
\providecommand{\MRhref}[2]{%
  \href{http://www.ams.org/mathscinet-getitem?mr=#1}{#2}
}
\providecommand{\href}[2]{#2}

\end{document}